\documentclass{amsart}
\usepackage{amsmath,amssymb,amsthm,enumerate}
\newtheorem{definition}{Definition}[section]
\newtheorem{proposition}[definition]{Proposition}
\newtheorem{lemma}[definition]{Lemma}
\newtheorem{theorem}[definition]{Theorem}
\newtheorem{corollary}[definition]{Corollary}
\newtheorem{remark}[definition]{Remark}

\title{Polynomial conserved quantities of Lie applicable surfaces}
\author[F.E. Burstall et al.]{Francis E. Burstall, Udo Hertrich-Jeromin, Mason Pember and Wayne Rossman}

\address[F.~Burstall]{Department of Mathematical Sciences, University of Bath, Bath, BA2~7AY. United Kingdom.}
\email{f.e.burstall@bath.ac.uk}

\address[U.~Hertrich-Jeromin]{Vienna University of Technology,
Wiedner Hauptstra\ss e 8-10/104, A-1040 Vienna. Austria.}
\email{udo.hertrich-jeromin@tuwien.ac.at}

\address[M.~Pember]{Vienna University of Technology,
Wiedner Hauptstra\ss e 8-10/104, A-1040 Vienna. Austria.}
\email{mason@geometrie.tuwien.ac.at}

\address[W.~Rossman]{Department of Mathematics, Kobe University,
Rokko, Kobe 657-8501. Japan.}
\email{wayne@math.kobe-u.ac.jp}

\begin{document}
\maketitle

\begin{abstract}
Using the gauge theoretic approach for Lie applicable surfaces, we characterise certain subclasses of surfaces in terms of polynomial conserved quantities. These include isothermic and Guichard surfaces of conformal geometry and $L$-isothermic surfaces of Laguerre geometry. In this setting one can see that the well known transformations available for these surfaces are induced by the transformations of the underlying Lie applicable surfaces. We also consider linear Weingarten surfaces in this setting and develop a new B\"{a}cklund-type transformation for these surfaces. 
\end{abstract}

\section{Introduction}
In~\cite{D1911i,D1911ii,D1911iii}, Demoulin defined a class of surfaces satisfying the equation 
\begin{equation}
\label{eqn:introdem} 
\left(\frac{V}{U}\frac{\sqrt{E}}{\sqrt{G}}\frac{\kappa_{1,u}}{\kappa_{1} - \kappa_{2}}\right)_{v} + \epsilon^{2}  \left(\frac{U}{V}\frac{\sqrt{G}}{\sqrt{E}}\frac{\kappa_{2,v}}{\kappa_{1}-\kappa_{2}}\right)_{u}=0,
\end{equation}
given in terms of curvature line coordinates $(u,v)$, where $U$ is a function of $u$, $V$ is a function of $v$, $\epsilon\in\{0,1,i\}$, $E$ and $G$ denote the usual coefficients of the first fundamental form and $\kappa_{1}$ and $\kappa_{2}$ denote the principal curvatures. In the case that $\epsilon\neq 0$, one calls these surfaces $\Omega$-surfaces and if $\epsilon=0$ we call them $\Omega_{0}$-surfaces. Together, $\Omega$- and $\Omega_{0}$-surfaces form the applicable surfaces of Lie sphere geometry (see~\cite{B1929}). By using the hexaspherical coordinate model of Lie~\cite{L1872} it is shown in~\cite{MN2006} that these surfaces are the deformable surfaces of Lie sphere geometry. This gives rise to a gauge theoretic approach for these surfaces which is developed in~\cite{C2012i}. That is, the definition of Lie applicable surfaces is equated to the existence of a certain 1-parameter family of flat connections. This approach lends itself well to the transformation theory of these surfaces. 
The gauge theoretic approach is explored further in~\cite{P2016}, for which this paper is meant as a sequel.

In~\cite{BS2012, S2008} a gauge theoretic approach for isothermic surfaces in M\"obius geometry is developed. By considering polynomial conserved quantities of the arising 1-parameter family of flat connections, one can characterise familiar subclasses of surfaces in certain space forms. For example, constant mean curvature surfaces in space forms are characterised by the existence of linear conserved quantities~\cite{BCwip,BS2012}. By then applying the transformation theory of the underlying isothermic surface, one obtains transformations for these subclasses. In this paper we apply this framework to Lie applicable surfaces. For example, we show that isothermic surfaces, Guichard surfaces and $L$-isothermic surfaces are $\Omega$-surfaces admitting a linear conserved quantity. This is particularly beneficial to the study of the transformations of these surfaces. For example, we will show that the Eisenhart transformation for Guichard surfaces (see~\cite{E1914}), which was given a conformally invariant treatment in~\cite{BCwip}, is induced by the Darboux transformation of the underlying $\Omega$-surface. One can also show (see~\cite{P2015}) that the special $\Omega$-surfaces of~\cite{E1916} can be characterised as $\Omega$-surfaces admitting quadratic conserved quantities, however we will not explore this further in this paper. 

In~\cite{BHR2010, BHR2012} linear Weingarten surfaces in space forms are characterised as Lie applicable surfaces whose isothermic sphere congruences take values in certain sphere complexes. In this paper we shall review this theory from the viewpoint of polynomial conserved quantities. We shall see that non-tubular linear Weingarten surfaces in space forms are $\Omega$-surfaces admitting a 2-dimensional vector space of linear conserved quantities, whereas tubular linear Weingarten surfaces are $\Omega_{0}$-surfaces admitting a constant conserved quantity. By using this approach we obtain a new B\"{a}cklund-type transformation for linear Weingarten surfaces. 

\textit{Acknowledgements.} We would like to thank G. Szewieczek for reading through this paper and providing many useful comments. This work has been partially supported by the Austrian Science Fund (FWF) through the research project P28427-N35 ``Non-rigidity and Symmetry breaking" as well as
by FWF and the Japan Society for the Promotion of Science (JSPS) through the FWF/JSPS Joint Project grant I1671-N26  ``Transformations and Singularities". The fourth author was also supported by the two JSPS grants Grant-in-Aid for Scientific Research (C) 15K04845 and (S) 24224001 (PI: M.-H. Saito).

\section{Preliminaries}
Given a vector space $V$ and a manifold $\Sigma$, we shall denote by $\underline{V}$ the trivial bundle $\Sigma\times V$. Given a vector subbundle $W$ of $\underline{V}$, we define the derived bundle of $W$, denoted $W^{(1)}$, to be the subset of $\underline{V}$ consisting of the images of sections of $W$ and derivatives of sections of $W$ with respect to the trivial connection on $\underline{V}$. In this paper, most of the derived bundles that appear will be vector subbundles of the trivial bundle, but in general this is not always the case as, for example, the rank of the derived bundle may not be constant over $\Sigma$. 

Throughout this paper we shall be considering the pseudo-Euclidean space $\mathbb{R}^{4,2}$, i.e., a 6-dimensional vector space equipped with a non-degenerate symmetric bilinear form $(\,,\,)$ of signature $(4,2)$. Let $\mathcal{L}$ denote the lightcone of $\mathbb{R}^{4,2}$. The orthogonal group $\textrm{O}(4,2)$ acts transitively on $\mathcal{L}$. The lie algebra $\mathfrak{o}(4,2)$ of $\textrm{O}(4,2)$ is well known to be isomorphic to the exterior algebra $\wedge^{2}\mathbb{R}^{4,2}$ via the identification 
\[ a\wedge b\, (c) = (a,c)b - (b,c)a,\]
for $a,b,c\in\mathbb{R}^{4,2}$. We shall frequently use this fact throughout this paper. 

By $\odot$ we shall denote the symmetric product on $\mathbb{R}^{4,2}$ that for $a,b\in \mathbb{R}^{4,2}$ gives $a\odot b = \frac{1}{2}(a\otimes b + b\otimes a) \in S^{2}(\mathbb{R}^{4,2})$, i.e., 
\[
a\odot b\,(v) = \frac{1}{2}((a,v)b + (b,v)a)\quad\text{and}\quad
a\odot b\,(v,w) = \frac{1}{2}((a,v)(b,w) + (a,w)(b,v)), 
\]
for $v,w\in \mathbb{R}^{4,2}$. 

Given a manifold $\Sigma$, we define the following product of two vector-valued 1-forms $\omega_{1},\omega_{2}\in\Omega^{1}(\underline{\mathbb{R}}^{4,2})$:
\[ \omega_{1}\curlywedge \omega_{2}(X,Y) := \omega_{1}(X)\wedge \omega_{2}(Y) - \omega_{1}(Y)\wedge \omega_{2}(X),\]
for $X,Y\in\Gamma T\Sigma$. Hence, $\omega_{1}\curlywedge \omega_{2}$ is a $2$-form taking values in $\wedge^{2}\underline{\mathbb{R}}^{4,2}$. Notice that $\omega_{1}\curlywedge \omega_{2} = \omega_{2}\curlywedge \omega_{1}$. 

Recall that we also have the following product for two $\mathfrak{so}(4,2)$-valued 1-forms $A,B\in \Omega^{1}(\mathfrak{so}(4,2))$:
\[ [A\wedge B](X,Y) = [A(X),B(Y)]-[A(Y),B(X)],\]
for $X,Y\in \Gamma T\Sigma$.

\subsection{Legendre maps}
Let $\mathcal{Z}$ denote the Grassmannian of isotropic 2-dimensional subspaces of $\mathbb{R}^{4,2}$.
Suppose that $\Sigma$ is a 2-dimensional manifold and let $f:\Sigma\to \mathcal{Z}$ be a smooth map. By viewing $f$ as a 2-dimensional subbundle of the trivial bundle $\underline{\mathbb{R}}^{4,2}$, we may define a tensor, analogous to the solder form defined in~\cite{BC2004,BR1990}, 
\[ \beta: T\Sigma \to Hom(f,f^{(1)}/f),\quad X\mapsto (\sigma \mapsto d_{X}\sigma \, mod\, f).\]
In accordance with~\cite[Theorem 4.3]{C2008} we have the following definition:

\begin{definition}
$f$ is a Legendre map if $f$ satisfies the contact condition, $f^{(1)}\le f^{\perp}$, and the immersion condition, $\ker\beta =\{0\}$. 
\end{definition}

\begin{remark}
\label{rem:legder}
The contact and immersion conditions together imply that $f^{(1)}=f^{\perp}$ (see~\cite{P1985}). 
\end{remark}

Note that $f^{\perp}/f$ is a rank $2$ subbundle of $\underline{\mathbb{R}}^{4,2}/f$ that inherits a positive definite metric from $\mathbb{R}^{4,2}$. 

\begin{definition}
\label{def:curvsph}
Let $p\in\Sigma$. Then a 1-dimensional subspace $s(p)\le f(p)$ is a curvature sphere of $f$ at $p$ if there exists a non-zero subspace $T_{s(p)}\le T_{p}\Sigma$ such that $\beta(T_{s(p)})s(p) = 0$. We call the maximal such $T_{s(p)}$ the curvature space of $s(p)$. 
\end{definition}

It was shown in~\cite{P1985} that at each point $p$ there is either one or two curvature spheres. We say that $p$ is an \textit{umbilic point of $f$} if there is exactly one curvature sphere $s(p)$ at $p$ and in that case $T_{s(p)}=T_{p}\Sigma$. 

\begin{lemma}
\label{lem:umbilic}
Suppose that $q\in (\mathbb{R}^{4,2})^{\times}$ such that $q\in \Gamma f$. Then $f$ is totally umbilic.  
\end{lemma}
\begin{proof}
If $\mathfrak{q}\in(\mathbb{R}^{4,2})^{\times}$ then $d\mathfrak{q}=0$. Therefore $\mathfrak{q}\in\Gamma f$ implies that $s:=\langle \mathfrak{q}\rangle$ is a curvature sphere congruence of $f$ with curvature subbundle $T_{s}=T\Sigma$. By the immersion condition of $f$ this is the only curvature sphere congruence of $f$ and thus $f$ is totally umbilic. 
\end{proof}

Away from umbilic points we have that the curvature spheres form two rank 1 subbundles $s_{1},s_{2}\le f$ with respective curvature subbundles $T_{1}=\bigcup_{p\in \Sigma}T_{s_{1}(p)}$ and $T_{2}=\bigcup_{p\in \Sigma}T_{s_{2}(p)}$. We then have that $f=s_{1}\oplus s_{2}$ and $T\Sigma = T_{1}\oplus T_{2}$. A conformal structure $c$ is induced on $T\Sigma$ as the set of all indefinite metrics whose null lines are $T_{1}$ and $T_{2}$.

\subsection{Symmetry breaking}
\label{sec:symbreak}
In~\cite{C2008} a modern account is given of how one breaks symmetry from Lie geometry to space form geometry and how $\textrm{O}(4,2)$ is a double cover for the group of Lie sphere transformations. These are the transformations that map oriented spheres to oriented spheres and preserve the oriented contact of spheres. In this subsection we shall recall the process of symmetry breaking. 

Firstly, we have the following technical result regarding projections of Legendre maps:
\begin{lemma}
\label{lem:totumb}
Suppose that $f:\Sigma\to \mathcal{Z}$ is a Legendre map and $\mathfrak{q}\in\mathbb{R}^{4,2}\backslash\{0\}$. Then
\begin{enumerate}
\item if $\mathfrak{q}$ is timelike then $f$ never belongs to $\langle\mathfrak{q}\rangle^{\perp}$,
\item  if $\mathfrak{q}$ is spacelike then the set of points $p\in\Sigma$ where $f(p)\le \langle\mathfrak{q}\rangle^{\perp}$ is a closed set with empty interior,
\item $f\le \langle \mathfrak{q}\rangle^{\perp}$ if and only if $\mathfrak{q}\in\Gamma f$, in which case $f$ is totally umbilic.
\end{enumerate}
\end{lemma}
\begin{proof}
If $\mathfrak{q}$ is timelike then $\langle\mathfrak{q}\rangle^{\perp}$ has signature $(4,1)$ and cannot contain the 2-dimensional lightlike subspace $f(p)$ for each $p\in\Sigma$. 

Suppose that $\mathfrak{q}$ is spacelike and that on some open subset $U\subset \Sigma$,  $f\le \langle \mathfrak{q}\rangle^{\perp}$. Without loss of generality, assume that $U=\Sigma$. Then this implies that $f^{(1)}\le \langle\mathfrak{q}\rangle^{\perp}$ and $\mathfrak{q}\in\Gamma f^{\perp}$. Hence, $f^{(1)}\neq f^{\perp}$, contradicting Remark~\ref{rem:legder}.

Therefore, if $f\le \langle \mathfrak{q}\rangle^{\perp}$ then the only possibility left to consider is that $\mathfrak{q}$ is lightlike. Then since the maximal lightlike subspaces of $\mathbb{R}^{4,2}$ are  2-dimensional, $\mathfrak{q}\in\Gamma f^{\perp}$ if and only if $\mathfrak{q}\in\Gamma f$. By Lemma~\ref{lem:umbilic} this is the case only if $f$ is totally umbilic. 
\end{proof}

We shall often refer to a non-zero vector $\mathfrak{q}\in\mathbb{R}^{4,2}$ as a \textit{sphere complex}. As Lemma~\ref{lem:totumb} shows, for a Legendre map $f:\Sigma\to\mathcal{Z}$, generically $f\cap\langle \mathfrak{q}\rangle^{\perp}$ defines a rank 1 subbundle of $f$. 

\subsubsection{Conformal geometry}
\label{subsubsec:conf}
Let $\mathfrak{p}\in\mathbb{R}^{4,2}$ such that $\mathfrak{p}$ is not lightlike. If $\mathfrak{p}$ is timelike then $\langle \mathfrak{p}\rangle^{\perp} \cong \mathbb{R}^{4,1}$ and defines a Riemannian conformal geometry. If $\mathfrak{p}$ is spacelike then $\langle \mathfrak{p}\rangle^{\perp} \cong \mathbb{R}^{3,2}$ and defines a Lorentzian conformal geometry. We consider elements of 
\[\mathbb{P}(\mathcal{L}\cap\langle\mathfrak{p}\rangle^{\perp})\]
to be points and refer to $\mathfrak{p}$ as a point sphere complex. 

\begin{remark}
In the case that $\mathfrak{p}$ is timelike, $\mathbb{P}(\mathcal{L}\cap\langle\mathfrak{p}\rangle^{\perp})$ is the conformal $3$-sphere (see \cite{H2003}).
\end{remark} 

The elements of $\mathbb{P}(\mathcal{L}\backslash \langle\mathfrak{p}\rangle^{\perp})$ give rise to spheres in the following way: suppose that $s\in \mathbb{P}(\mathcal{L}\backslash \langle\mathfrak{p}\rangle^{\perp})$. Now $s\oplus\langle\mathfrak{p}\rangle$ is a $(1,1)$-plane and thus 
\[ V:= (s\oplus\langle\mathfrak{p}\rangle)^{\perp}\]
is a $(3,1)$-plane. The projective lightcone of $V$ is then diffeomorphic to $\mathbb{S}^{2}$ and we thus identify $V$ with a sphere in $\mathbb{P}(\mathcal{L}\cap\langle\mathfrak{p}\rangle^{\perp})$. 

Conversely, suppose that $V\le \langle\mathfrak{p}\rangle^{\perp}$ is a $(3,1)$-plane. Then $V^{\perp}$ is a $(1,1)$-plane in $\mathbb{R}^{4,2}$ containing $\mathfrak{p}$ and we identify the two null lines of $V^{\perp}$ with the sphere defined by $V$ with opposite orientations.

\begin{remark}
Those Lie sphere transformations that fix the point sphere complex are the conformal transformations of $\langle \mathfrak{p}\rangle^{\perp}$. 
\end{remark}

As is standard in conformal geometry (see, for example,~\cite{H2003}), we may break symmetry further by choosing a vector $\mathfrak{q}\in\langle\mathfrak{p}\rangle^{\perp}$. Then 
\[ \mathfrak{Q}^{3}:=\{y\in\mathcal{L}: \,(y,\mathfrak{q})=-1,\, (y,\mathfrak{p})=0\}\]
is isometric to a space form with sectional curvature $\kappa = -|\mathfrak{q}|^{2}$. If we assume that $|\mathfrak{p}|^{2}=\pm 1$, then 
\[ \mathfrak{P}^{3}:=\{y\in\mathcal{L}:\, (y,\mathfrak{q})=0,\, (y,\mathfrak{p})=-1\}\]
can be identified (see~\cite{H2003}) with the space of hyperplanes (complete, totally geodesic hypersurfaces) in this space form. 

Suppose that $f:\Sigma\to \mathcal{Z}$ is a Legendre map. Then, by Lemma~\ref{lem:totumb}, on a dense open subset of $\Sigma$, $\Lambda:=f\cap\langle \mathfrak{p}\rangle^{\perp}$ is a rank 1 subbundle of $f$. Using the identification of $\wedge^{2}\mathbb{R}^{4,2}$ with the skew-symmetric endomorphisms on $\mathbb{R}^{4,2}$, we have for any $\tau\in \Gamma (\wedge^{2}f)$ that $\tau\mathfrak{p} \in \Gamma f$ and, since $\tau$ is skew-symmetric, $\tau\mathfrak{p}\perp\mathfrak{p}$. Hence, 
\[ \Lambda= (\wedge^{2}f) \mathfrak{p}.\]
Away from points where $\Lambda\perp\mathfrak{q}$, we have that for any nowhere zero $\tau\in\Gamma(\wedge^{2}f)$,
\[ \mathfrak{f}:= - \frac{\tau \mathfrak{p}}{(\tau\mathfrak{p},\mathfrak{q})}\quad\text{and}\quad 
\mathfrak{t}:= - \frac{\tau \mathfrak{q}}{(\tau\mathfrak{q},\mathfrak{p})}\]
are the projections of $f$ into $\mathfrak{Q}^{3}$ and $\mathfrak{P}^{3}$, respectively. We can then write $f = \langle \mathfrak{f},\mathfrak{t}\rangle$. 

\begin{definition}
We call $\mathfrak{f}$ the space form projection of $f$ and $\mathfrak{t}$ the tangent plane\footnote{Note that ``plane'' here means totally geodesic hypersurface in the space form $\mathfrak{Q}^{3}$.} congruence of $f$. 
\end{definition}

One can easily see that:

\begin{lemma}
\label{lem:spfwedge}
The space form projection of $f$ into $\mathfrak{Q}^{3}$ exists at $p\in\Sigma$ if and only if the kernel of the linear map
\[ \wedge^{2}f\to\mathbb{R}, \quad \tau\mapsto (\tau\mathfrak{p},\mathfrak{q}) \]
is trivial at $p$. 
\end{lemma}

Away from umbilic points, suppose that $(u,v)$ are curvature line coordinates for $\mathfrak{f}$. Then by Rodrigues' equations we have that 
\[\mathfrak{t}_{u}+\kappa_{1}\mathfrak{f}_{u} =0 = \mathfrak{t}_{v}+\kappa_{2}\mathfrak{f}_{v},\]
where $\kappa_{1}$ and $\kappa_{2}$ are the principal curvatures of $\mathfrak{f}$. Therefore, 
\[ s_{1}:= \langle \mathfrak{t} +\kappa_{1}\mathfrak{f}\rangle \quad \text{and}\quad s_{2} :=  \langle \mathfrak{t} +\kappa_{2}\mathfrak{f}\rangle\]
are curvature spheres of $f$ with respective curvature subbundles $T_{1}:= \left\langle \frac{\partial}{\partial u}\right\rangle$ and $T_{2}:= \left\langle \frac{\partial}{\partial v}\right\rangle$.

\subsubsection{Laguerre geometry}
\label{subsec:laguerre}
In this subsection we shall recall the correspondence given in~\cite{C2008} between Lie sphere geometry and Laguerre geometry. Let $\mathfrak{q}_{\infty}\in\mathcal{L}$ and define $U:= \mathbb{P}(\mathcal{L})\backslash \langle\mathfrak{q}_{\infty}\rangle^{\perp}$. Then $(E,\psi)$ with 
\[ E:= \{y\in\mathcal{L}: (y,\mathfrak{q}_{\infty})=-1\}\quad \text{and} \quad \psi:E\to U, \quad y\mapsto [y] \]
defines an affine chart for $U$. Choosing $\mathfrak{q}_{0}\in\mathcal{L}$ such that $(\mathfrak{q}_{0},\mathfrak{q}_{\infty})=-1$, we have that $\langle \mathfrak{q}_{0},\mathfrak{q}_{\infty}\rangle^{\perp}\cong\mathbb{R}^{3,1}$. We may then define the orthogonal projection
\[ \pi:\mathbb{R}^{4,2}\to \langle \mathfrak{q}_{0},\mathfrak{q}_{\infty}\rangle^{\perp} , \quad y\mapsto y + (y,\mathfrak{q}_{\infty})\mathfrak{q}_{0} + (y,\mathfrak{q}_{0})\mathfrak{q}_{\infty}.\]
Then $\pi\circ\psi^{-1}$ defines an isomorphism between $U$ and $\langle \mathfrak{q}_{0},\mathfrak{q}_{\infty}\rangle^{\perp}$. We thus identify points in $U$ with points in $\mathbb{R}^{3,1}$. Now let $W:= \mathbb{P}(\mathcal{L}\cap \langle\mathfrak{q}_{\infty}\rangle^{\perp})\backslash \langle \mathfrak{q}_{\infty}\rangle$. Then $\pi$ identifies $W$ with the projective lightcone of $\langle \mathfrak{q}_{0},\mathfrak{q}_{\infty}\rangle^{\perp}$ and thus $\mathbb{P}(\mathcal{L}^{3})$, where $\mathcal{L}^{3}$ is the lightcone of $\mathbb{R}^{3,1}$. Therefore, we identify $W$ with null directions in $\mathbb{R}^{3,1}$. We define $\langle\mathfrak{q}_{\infty}\rangle$ to be the improper point of Laguerre geometry. 

Under this correspondence, contact elements in $\mathbb{R}^{4,2}$ are then identified with affine null lines in $\mathbb{R}^{3,1}$, i.e., for $z\in\mathbb{R}^{3,1}$ and $l\in\mathbb{P}(\mathcal{L}^{3})$
\[ [z,l] := \{ z+ v: v\in l\}.\]

By choosing a point sphere complex $\mathfrak{p}\in\langle \mathfrak{q}_{0},\mathfrak{q}_{\infty}\rangle^{\perp}$ with $|\mathfrak{p}|^{2}=-1$, we have that 
\[ \langle\mathfrak{q}_{0},\mathfrak{q}_{\infty},\mathfrak{p}\rangle^{\perp}\cong\mathbb{R}^{3}.\]
One identifies points in $\mathbb{R}^{3,1}$ with oriented spheres (including point spheres, but not oriented planes) in $\mathbb{R}^{3}$ in the following way: a sphere centred at $x\in\mathbb{R}^{3}$ with signed radius $r\in\mathbb{R}$ is identified with the point
\[ x+ r\mathfrak{p}\in\mathbb{R}^{3,1}.\]
This is classically known as isotropy projection~\cite{B1929,C2008}. We then have that null lines in $\mathbb{R}^{3,1}$ correspond to pencils of spheres in $\mathbb{R}^{3}$ in oriented contact with each other and isotropic planes in $\mathbb{R}^{3,1}$ are identified with oriented planes in $\mathbb{R}^{3}$. 

It was shown in~\cite{C2008} that the Lie sphere transformations $A\in\textrm{O}(4,2)$ that preserve the improper point $\langle\mathfrak{q}_{\infty}\rangle$ are identified under this correspondence with the affine Laguerre transformations of $\mathbb{R}^{3,1}$, that is, the identity component of the group $\mathbb{R}^{4}\rtimes \textrm{O}(3,1)$. In terms of transformations of $\mathbb{R}^{3}$, this group consists of the Lie sphere transformations that map oriented planes to oriented planes. 

Defining 
\[ \mathfrak{Q}^{3}:= \{ y\in\mathcal{L}: (y,\mathfrak{q}_{\infty})=-1, (y,\mathfrak{p})=0\},\]
we have that $\pi|_{\mathfrak{Q}^{3}}$ is an isometry between $\mathfrak{Q}^{3}$ and $\langle\mathfrak{q}_{0},\mathfrak{q}_{\infty},\mathfrak{p}\rangle^{\perp}$ and this restricts to the usual Euclidean projection in the conformal geometry defined by $\langle\mathfrak{p}\rangle^{\perp}$, see~\cite{BS2012,S2008,H2003}. 

\subsection{Lie applicable surfaces}

\begin{definition}[{\cite[Definition 3.1]{P2016}}] 
\label{def:lieapp}
We say that $f$ is a Lie applicable surface if there exists a closed $\eta\in \Omega^{1}(f\wedge f^{\perp})$ such that $[\eta\wedge \eta]=0$ and the quadratic differential $q$ defined by  
\[ q(X,Y)= tr(\sigma\mapsto \eta(X)d_{Y}\sigma : f\to f)\]
is non-zero. Furthermore, if $q$ is non-degenerate (respectively, degenerate) on a dense open subset of $\Sigma$ we say that $f$ is an $\Omega$-surface ($\Omega_{0}$-surface). 
\end{definition}

Given a closed $\eta\in \Omega^{1}(f\wedge f^{\perp})$, we have for any $\tau\in \Gamma(\wedge^{2}f)$ that $\tilde{\eta} := \eta-d\tau$ is a new closed 1-form taking values in $\Omega^{1}(f\wedge f^{\perp})$. We then say that $\tilde{\eta}$ and $\eta$ are \textit{gauge equivalent} and this yields an equivalence relation on closed 1-forms with values in $f\wedge f^{\perp}$. We call the equivalence class 
\[ [\eta] :=\{\eta -d\tau: \tau\in \Gamma(\wedge^{2}f)\}\]
the \textit{gauge orbit of $\eta$}. As shown in \cite[Corollary 3.3]{P2016}, $q$ is well defined on gauge orbits, i.e., if $\eta$ and $\tilde{\eta}$ are gauge equivalent then $q=\tilde{q}$, for their respective quadratic differentials.

Let us assume that $f$ is umbilic-free. Then there are two distinct curvature sphere congruences $s_{1}$ and $s_{2}$ with respective curvature subbundles $T_{1}$ and $T_{2}$. 

\begin{proposition}[{\cite[Proposition 3.4]{P2016}}]
\label{prop:mceqn}
For an umbilic-free Legendre map $f$, $\eta\in \Omega^{1}(f\wedge f^{\perp})$ is closed if and only if $\eta$ satisfies the Maurer Cartan equation, i.e., $d\eta+\frac{1}{2}[\eta\wedge \eta]=0$. In this case, $\eta(T_{i})\le f\wedge f_{i}$ and $[\eta\wedge \eta]=0$. 
\end{proposition}

By Proposition~\ref{prop:mceqn}, $\eta$ being closed implies that $[\eta\wedge \eta]=0$. Therefore, we may drop the condition that $[\eta\wedge \eta]=0$ from Definition~\ref{def:lieapp} when we are working with umbilic-free Legendre maps. 

One has a splitting of the trivial bundle called the \textit{Lie cyclide splitting}:
\[ \underline{\mathbb{R}}^{4,2}= S_{1}\oplus_{\perp} S_{2},\]
where
\[ S_{1} = \langle \sigma_{1}, d_{Y}\sigma_{1}, d_{Y}d_{Y}\sigma_{1}\rangle \quad \text{and}\quad S_{2} = \langle \sigma_{2}, d_{X}\sigma_{2}, d_{X}d_{X}\sigma_{2}\rangle,\]
for $\sigma_{1}\in \Gamma s_{1}$, $\sigma_{2}\in \Gamma s_{2}$, $X\in \Gamma T_{1}$ and $Y\in \Gamma T_{2}$. Since we may identify $\mathfrak{o}(4,2)$ with $\wedge^{2}\mathbb{R}^{4,2}$, we have a splitting\footnote{This is the Cartan decomposition for the symmetric space of Dupin cyclides.} 
\[ \underline{\mathfrak{o}(4,2)} = \mathfrak{h}\oplus \mathfrak{m},\]
where 
\[ \mathfrak{h} = S_{1}\wedge S_{1}\oplus S_{2}\wedge S_{2} \quad \text{and}\quad \mathfrak{m} = S_{1}\wedge S_{2}.\]
Therefore,we may split a closed 1-form $\eta$ into $\eta = \eta_{\mathfrak{h}} + \eta_{\mathfrak{m}}$, accordingly. In~\cite[Definition 3.8]{P2016} it is shown that there is a unique member of the gauge orbit of $\eta$ that satisfies $\eta_{\mathfrak{m}}\in \Omega^{1}(\wedge^{2} f)$. We call this unique member the \textit{middle potential} and denote it by $\eta^{mid}$.

\textbf{Assumption:} for the rest of this paper we will make the assumption that the signature of the quadratic differential $q$ is constant over all of $\Sigma$. 

From Proposition~\ref{prop:mceqn}, one can deduce that $q\in \Gamma ((T_{1}^{*})^{2}\oplus (T_{2}^{*})^{2})$. Therefore, after possibly rescaling $q$ by $\pm 1$ and switching $T_{1}$ and $T_{2}$, we may write 
\[ q = (d\sigma_{1},d\sigma_{1}) - \epsilon^{2}(d\sigma_{2},d\sigma_{2}),\]
for unique (up to sign) lifts of the curvature sphere congruences $\sigma_{1}\in \Gamma s_{1}$ and $\sigma_{2}\in \Gamma s_{2}$. The middle potential is then given by 
\begin{equation}
\label{eqn:middle}
\eta^{mid} = \sigma_{1}\wedge\star d\sigma_{1} + \epsilon^{2} \sigma_{2}\wedge \star d\sigma_{2},
\end{equation}
where $\star$ is the Hodge-star operator of the conformal structure $c$ for which the curvature directions on $T\Sigma$ are null. One finds that $q$ is divergence-free with respect to $c$, i.e., in terms of curvature line coordinates $u$ and $v$, there exist functions $U$ of $u$ and $V$ of $v$ such that
\begin{equation} 
\label{eqn:qdivfree}
q = -\epsilon^{2}U^{2}du^{2} + V^{2}dv^{2}.
\end{equation}
When one projects to a space form, where the space form projection immerses, one finds that Demoulin's equation 
\[ \left(\frac{V\sqrt{E}}{U\sqrt{G}} \frac{\kappa_{1,u}}{\kappa_{1}-\kappa_{2}}\right)_{v} +\epsilon^{2}\left(\frac{U\sqrt{G}}{V\sqrt{E}} \frac{\kappa_{2,v}}{\kappa_{1}-\kappa_{2}} \right)_{u}=0\]
is satisfied. 

By gauging $\eta^{mid}$ by $\pm \epsilon \sigma_{1}\wedge \sigma_{2}$, we obtain closed 1-forms 
\[ \eta^{\pm} := (\sigma_{1}\pm \epsilon \sigma_{2})\wedge \star d(\sigma_{1}\pm \epsilon \sigma_{2}).\]
Thus, $s^{\pm}:= \langle \sigma_{1}\pm \epsilon \sigma_{2}\rangle$ are isothermic sphere congruences (see~\cite{BDPP2011, H2003}). There then exist (unique up to constant reciprocal rescaling) lifts $\sigma^{\pm}\in \Gamma s^{\pm}$ such that 
\begin{equation*}
\eta^{\pm} = \sigma^{\pm}\wedge d\sigma^{\mp}.
\end{equation*}
We call these lifts the \textit{Christoffel dual lifts} of $s^{\pm}$. In terms of these lifts the middle potential has the form:
\begin{equation}
\label{eqn:midchris}
\eta^{mid} = \frac{1}{2}\left(\sigma^{+}\wedge d\sigma^{-} + \sigma^{-}\wedge d\sigma^{+} \right).
\end{equation}

\subsection{Transformations of Lie applicable surfaces}
\label{subsec:trafos}
The transformation theory for Lie applicable surfaces was developed in~\cite{C2012i} and was further explored in~\cite{P2016}. In this section we shall review some of this theory. The richness of the transformation theory of Lie applicable surfaces follows from the following result:

\begin{theorem}[{\cite[Lemma 4.2.6]{C2012i}}]
Suppose that $\eta\in \Omega^{1}(f\wedge f^{\perp})$ is closed and $[\eta\wedge \eta]=0$. Then $\{d+t\eta\}_{t\in \mathbb{R}}$ is a 1-parameter family of flat metric connections. 
\end{theorem} 

Suppose now that $\tilde{\eta}:=\eta - d\tau\in[\eta]$, for some $\tau \in \Gamma(\wedge^{2}f)$.

\begin{lemma}[{\cite[Lemma 4.5.1]{C2012i}}] 
\label{lem:flatgorb} $d+t\tilde{\eta} = \exp(t\tau)\cdot (d+t\eta).$
\end{lemma}

In the case that we are using the middle potential, $\eta^{mid}$, we shall refer to the 1-parameter family of connections $\{d+t\eta^{mid}\}_{t\in\mathbb{R}}$ as the \textit{middle pencil of connections}, or for brevity, the \textit{middle pencil}. 

\subsubsection{Calapso transforms}
\label{subsubsec:caltrafos}
For each $t\in \mathbb{R}$ and gauge potential $\eta$, since $d+t\eta$ is a flat metric connection, there exists a local orthogonal trivialising gauge transformation $T(t):\Sigma\to O(4,2)$, that is, 
\[ T(t)\cdot (d+t\eta) = d.\]

\begin{definition}
$f^{t}:= T(t)f$ is called a Calapso transform of $f$. 
\end{definition}

By Lemma~\ref{lem:flatgorb}, if $\tilde{\eta} = \eta - d\tau$ is in the gauge orbit of $\eta$, then the corresponding local orthogonal trivialising gauge transformations of $d+t\tilde{\eta}$ are given by 
\[ \widetilde{T}(t) = T(t)\exp(-t\tau).\]
Since $(\wedge^{2}f)f=0$, we have that the Calapso transforms are well defined on the gauge orbit. 

In~\cite[Theorem 4.4]{P2016} it is shown that $\eta^{t}:= Ad_{T(t)}\cdot\eta$ is a closed 1-form taking values in $f^{t}\wedge (f^{t})^{\perp}$. Furthermore, $[\eta^{t}\wedge \eta^{t}]=0$ and $q^{t}=q$. Thus we have the following theorem:

\begin{theorem}
\label{thm:calform}
Calapso transforms are Lie applicable surfaces. 
\end{theorem}

In fact, this 1-parameter family of Lie applicable surfaces arises because Lie applicable surfaces are the deformable surfaces of Lie sphere geometry (see~\cite{MN2006}). 

\begin{proposition}[{\cite[Proposition 4.5]{P2016}}]
\label{prop:calgauge}
For any $s\in\mathbb{R}$, 
\[ d+s\eta^{t} = T(t)\cdot (d+(s+t)\eta).\]
Therefore 
\[ T^{t}(s) = T(s+t)T^{-1}(t)\]
are the local trivialising orthogonal gauge transformations of $d+s\eta^{t}$. 
\end{proposition}

\subsubsection{Darboux transforms}
\label{subsubsec:darbtrafo}
Fix $m\in\mathbb{R}^{\times}$ and let $\eta$ be any gauge potential. Since $d^{m}:=d+m\eta$ is a flat connection, it has many parallel sections. Suppose that $\hat{s}$ is a null rank 1 parallel subbundle of $d^{m}$ such that $\hat{s}$ is nowhere orthogonal to the curvature sphere congruences of $f$. Let $s_{0}: =  \hat{s}^{\perp}\cap f$ and let $\hat{f}:= s_{0}\oplus \hat{s}$. 

\begin{definition}
$\hat{f}$ is a Darboux transform of $f$ with parameter $m$. 
\end{definition}

If $\tilde{\eta}=\eta - d\tau$, then by Lemma~\ref{lem:flatgorb}, we have that $\hat{s}':=\exp(m\tau)\hat{s}$ is a parallel subbundle of $d+m\tilde{\eta}$. However, $\hat{s}$ and $\hat{s}'$ determine the same $\hat{f}$. Thus, Darboux transforms are invariant of choice of gauge potential in the gauge orbit of $\eta$. 

It was shown in~\cite[Theorem 4.3.7, Proposition 4.3.8]{C2012i} that $\hat{f}$ is a Lie applicable surface, and $f$ is a Darboux transform of $\hat{f}$ with parameter $m$, i.e., for any gauge potential $\hat{\eta}\in \Omega^{1}(\hat{f}\wedge \hat{f}^{\perp})$, there exists a parallel subbundle $s\le f$ of $d+m\hat{\eta}$. Thus:

\begin{theorem}
Darboux transforms of Lie applicable surfaces are Lie applicable surfaces.
\end{theorem}

Recall from~\cite{BS2012,S2008,BDPP2011,C2012i} that for $L,\hat{L}\in \mathbb{P}(\mathcal{L})$ such that $L\not\perp \hat{L}$ and $t\in\mathbb{R}^{\times}$ we have an orthogonal transformation  
\[ \Gamma^{\hat{L}}_{L}(t)u = \left\{ \begin{array}{ll}
t\, u &\text{for $u\in \hat{L}$,}\\
\frac{1}{t}\, u &\text{for $u\in L$,}\\
u &\text{for $u\in (L\oplus \hat{L})^{\perp}$.}
\end{array}\right.\] 
In the case that $f$ and $\hat{f}$ are umbilic-free we have the following result regarding the middle pencils of the two surfaces:

\begin{proposition}[{\cite[Proposition 4.17, Theorem 4.19]{P2016}}]
\label{prop:darbmid}
Suppose that $f$ and $\hat{f}$ are umbilic-free Darboux transforms of each other with parameter $m$. 
Then
\begin{equation*} 
d+t\hat{\eta}^{mid} = \Gamma^{\hat{s}}_{s}(1-t/m)\cdot (d+t\eta^{mid}),
\end{equation*}
where $\hat{s}\le \hat{f}$ and $s\le f$ are the parallel subbundles of $d+m\eta^{mid}$ and $d+m\hat{\eta}^{mid}$, respectively, implementing these Darboux transforms. 
\end{proposition}

From~\cite{P2016} we also have the following proposition:

\begin{proposition}[{\cite[Proposition 4.14]{P2016}}]
\label{prop:darbfree}
Suppose that $\hat{f}$ is a Darboux transform of $f$ with parameter $m$ and let $l$ be any rank 2 subbundle of $f+\hat{f}$ with $l\cap s_{0}=\{0\}$. Then there exist gauge potentials $\eta\in\Omega^{1}(f\wedge f^{\perp})$ and $\hat{\eta}\in\Omega^{1}(\hat{f}\wedge \hat{f}^{\perp})$ such that $s:=f\cap l$ is a parallel subbundle of $d+m\hat{\eta}$ and $\hat{s}:=\hat{f}\cap l$ is a parallel subbundle of $d+m\eta$. 
\end{proposition}

In particular, this proposition shows that given any subbundle $\hat{s}\le\hat{f}$, one may choose a gauge potential $\eta$ such that $\hat{s}$ is a parallel subbundle of $d+m\eta$.

A pertinent question is ``how many Darboux transforms does a Lie applicable surface admit?" By using that 
\[ T(m)\cdot (d+m\eta) = d,\]
for every $m\in \mathbb{R}$, one deduces the following lemma:

\begin{lemma}
\label{lem:darbcons}
$\hat{s}$ is a null rank 1 parallel subbundle of $d+m\eta$ if and only if $\hat{s}= T^{-1}(m)\hat{L}$ for some constant $\hat{L}\in \mathbb{P}(\mathcal{L})$. 
\end{lemma}

Now, since $\mathbb{P}(\mathcal{L})$ is 4-dimensional, $d+m\eta$ admits a 4-parameter family of null rank 1 parallel subbundles. Since this holds for every $m\in\mathbb{R}$, we obtain the following answer to our question:

\begin{theorem}\cite{E1915,E1962}
\label{thm:darbnumb}
A Lie applicable surface admits a 5-parameter family of Darboux transforms. 
\end{theorem}

\subsection{Associate surfaces}\label{subsec:asssurf}
Let $\mathfrak{q}_{\infty}$ and $\mathfrak{p}$ be a space form vector and point sphere complex with $|\mathfrak{q}_{\infty}|^{2}=0$ and $|\mathfrak{p}|^{2}=-1$, i.e., 
\[ \mathfrak{Q}^{3}:=\{y\in\mathcal{L}:(y,\mathfrak{q}_{\infty})=-1, (y,\mathfrak{p})=0\}\]
has sectional curvature $\kappa=0$ and $\mathfrak{Q}^{3}\cong \mathbb{R}^{3}$. Then we may choose a null vector $\mathfrak{q}_{0}\in \langle \mathfrak{p}\rangle^{\perp}$ such that $(\mathfrak{q}_{0},\mathfrak{q}_{\infty})=-1$. Thus $\langle \mathfrak{q}_{\infty},\mathfrak{p},\mathfrak{q}_{0}\rangle^{\perp}\cong \mathbb{R}^{3}$ and we have an isometry 
\[\phi:\langle \mathfrak{q}_{\infty},\mathfrak{p},\mathfrak{q}_{0}\rangle^{\perp}\to \mathfrak{Q}^{3},\quad x\mapsto x + \mathfrak{q}_{0} + \frac{1}{2}(x,x)\mathfrak{q}_{\infty}.\]
We can use this to identify $\mathfrak{f}:= f\cap\mathfrak{Q}^{3}$ with a surface $x:\Sigma\to\mathbb{R}^{3}$. Let $n:\Sigma\to S^{2}$ denote the unit normal of $x$. We then have that $d\mathfrak{f} = dx + (dx,x)\mathfrak{q}_{\infty}$ and the tangent plane congruence of $\mathfrak{f}$ is given by $\mathfrak{t} = n + (n,x)\mathfrak{q}_{\infty} + \mathfrak{p}$. 

It was shown in~\cite[Section 5]{P2016} that there exists a 1-parameter family of closed 1-forms $\eta$ in the gauge orbit of $\eta^{mid}$ satisfying $(\eta\mathfrak{p},\mathfrak{q}_{\infty})=0$. We may then write 
\begin{equation}
\label{eqn:asseta}
\eta = \mathfrak{f}\wedge (dx^{D} + (dx^{D},x)\mathfrak{q}_{\infty}) + \mathfrak{t}\wedge (d\hat{x} + (d\hat{x}, x)\mathfrak{q}_{\infty}),
\end{equation}
where $x^{D}$ and $\hat{x}$ are Combescure transforms of $x$, i.e., $x^{D}$ and $\hat{x}$ have parallel curvature directions to $x$, such that the principal curvatures of the surfaces satisfy
\begin{equation}
\label{eqn:asscurv}
\frac{1}{\kappa_{1}\kappa^{D}_{2}} +  \frac{1}{\kappa_{2}\kappa^{D}_{1}} -  \frac{1}{\hat{\kappa}_{1}} - \frac{1}{\hat{\kappa}_{2}}=0.
\end{equation}
This shows that $\{x,x^{D},\hat{x},n\}$ forms a system of $O$-surfaces, see~\cite{KS2003}. Conversely, given such a system of surfaces satisfying~(\ref{eqn:asscurv}), one can check that $\eta$ defined in~(\ref{eqn:asseta}) is a closed 1-form, and thus $f$ is an $\Omega$-surface. 

We call $x^{D}$ an \textit{associate surface} of $x$ and $\hat{x}$ an \textit{associate Gauss map} of $x$. In~\cite[Theorem 5.4]{P2016} it was shown that an associate surface of an $\Omega$-surface is itself an $\Omega$-surface.

\section{Polynomial conserved quantities}
Suppose that $f:\Sigma\to \mathcal{Z}$ is a Lie applicable surface with family of flat connections $\{d^{t} = d+t\eta\}_{t\in\mathbb{R}}$. We now give a definition that is analogous to that of \cite{BS2012,S2008}:

\begin{definition}
A non-zero polynomial $p = p(t)\in\Gamma \underline{\mathbb{R}}^{4,2}[t]$ is called a polynomial conserved quantity of $\{d^{t}\}_{t\in\mathbb{R}}$ if $p(t)$ is a parallel section of $d^{t}$ for all $t\in\mathbb{R}$.
\end{definition}

The following lemma shows that the existence of polynomial conserved quantities is gauge invariant. Suppose that $\tilde{\eta}$ is in the gauge orbit of $\eta$ so that $\tilde{\eta} = \eta - d\tau$ for $\tau\in \Gamma (\wedge^{2} f)$. From Lemma~\ref{lem:flatgorb} we immediately get the following result:

\begin{lemma}
\label{lem:polygauge}
Suppose that $p$ is a polynomial conserved quantity of $d+t\eta$. Then $\tilde{p}(t) = \exp(t\tau)p(t)$ is a polynomial conserved quantity of $d+t\tilde{\eta}$ with $\tilde{p}(0) = p(0)$. 
\end{lemma}

Using an identical argument to \cite[Proposition 2.2]{BS2012}, one obtains the following lemma:

\begin{lemma}
Suppose that $p$ is a polynomial conserved quantity of $d^{t}$. Then the real polynomial $(p(t),p(t))$ has constant coefficients.
\end{lemma}

From now on we shall assume that $f$ is an umbilic-free $\Omega$-surface and assume that $\eta$ is the middle potential $\eta^{mid}$. 

\begin{proposition}
\label{prop:polymid}
Suppose that $p(t) =\sum\limits_{k=0}^{d}p_{k}t^{k}$ is a degree $d$ polynomial conserved quantity of $d+t\eta^{mid}$. Then 
\begin{enumerate}
\item $p_{0}$ is constant.
\item\label{item:polytop} For the Christoffel dual lifts $\sigma^{\pm}$, one has that $p_{d} = -(\sigma^{+}\odot\sigma^{-}) p_{d-1}$. Furthermore, $(\sigma^{\pm},p_{d-1})$ are constants. 
\item\label{item:polygauge} For any $\tau\in\Gamma (\wedge^{2}f)$, $\tilde{p}(t) = \exp(t\tau) p(t)$ has degree at most $d$ and the coefficient of $t^{d}$ is given by $p_{d} + \tau p_{d-1}$.  
\end{enumerate}                     
\end{proposition}
\begin{proof}
Consider the polynomial $(d+t\eta^{mid})p(t)$ whose coefficients take values in $\Omega^{1}(\underline{\mathbb{R}}^{4,2})$: 
\begin{equation}
\label{eqn:polysum}
 0 = (d+t\eta^{mid})p(t) = dp_{0} + \sum\limits_{k=1}^{d}t^{k}(dp_{k} + \eta^{mid} p_{k-1})+ t^{d+1}\eta^{mid}p_{d}. 
\end{equation}
Therefore $dp_{0}=0$ and thus $p_{0}$ is constant. Furthermore, $\eta^{mid}p_{d}=0$. Now by Equation~(\ref{eqn:midchris}), in terms of special lifts of the curvature spheres, the middle potential is given by 
\[ \eta^{mid} = \frac{1}{2}\left(\sigma^{+}\wedge d\sigma^{-} + \sigma^{-}\wedge d\sigma^{+}\right).\]
Thus, $\eta^{mid}p_{d}=0$ implies that
\begin{equation} 
\label{eqn:polytop}
0= (\sigma^{+},p_{d})d\sigma^{-} - (d\sigma^{-},p_{d})\sigma^{+} + (\sigma^{-},p_{d})d\sigma^{+} - (d\sigma^{+},p_{d})\sigma^{-}.
\end{equation}
One deduces that $(\sigma^{\pm},p_{d})=0$, as otherwise one would have that $\nu:=(\sigma^{+},p_{d})\sigma^{-}+ (\sigma^{-},p_{d})\sigma^{+}$ is a section of $f$ satisfying $d\nu\in \Omega^{1}(f)$, which contradicts that $f$ is umbilic-free. Furthermore, from (\ref{eqn:polytop}) we have that $(d\sigma^{\pm},p_{d})=0$. Therefore, since $f^{(1)} = f^{\perp}$, $p_{d}$ takes values in $(f^{\perp})^{\perp}= f$. Thus 
\[p_{d} = \lambda \sigma^{+} + \mu \sigma^{-}\]
for some smooth functions $\lambda$ and $\mu$. By~(\ref{eqn:polysum}), one has that $dp_{d}+\eta^{mid}p_{d-1}=0$. Therefore, modulo terms in $f$, one has that 
\[ \lambda d\sigma^{+} + \mu d\sigma^{-} + \frac{1}{2}((\sigma^{+},p_{d-1})d\sigma^{-} + (\sigma^{-},p_{d-1})d\sigma^{+}) = 0\, mod\, f.\]
Hence, $\lambda = -\frac{1}{2} (\sigma^{-},p_{d-1})$ and $\mu = - \frac{1}{2}(\sigma^{+},p_{d-1})$, as otherwise 
\[\nu:= (\lambda +\frac{1}{2} (\sigma^{-},p_{d-1}))\sigma^{+} + (\mu + \frac{1}{2}(\sigma^{+},p_{d-1}))\sigma^{-}\]
would define a section of $f$ satisfying $d\nu\in \Omega^{1}(f)$, contradicting that $f$ is umbilic-free. Returning to the equation $dp_{d}+\eta^{mid}p_{d-1}=0$ and evaluating the terms taking values in $f$, one has that 
\begin{equation}
\label{eqn:chrisconst}
-d(\sigma^{-},p_{d-1})\, \sigma^{+} - d(\sigma^{+},p_{d-1})\, \sigma^{-} -(d\sigma^{-},p_{d-1})\sigma^{+} - (d\sigma^{+},p_{d-1})\sigma^{-} =0.
\end{equation}
Now by~(\ref{eqn:polysum}), $dp_{d-1} + \eta p_{d-2}=0$, and thus $dp_{d-1}\in \Omega^{1}(f^{\perp})$. Hence, $d(\sigma^{\pm},p_{d-1})=(d\sigma^{\pm},p_{d-1})$ and thus~(\ref{eqn:chrisconst}) implies that $(\sigma^{\pm},p_{d-1})$ are constant. 

Since $p_{d}\in\Gamma f$, $\tau p_{d} = 0$ for any $\tau\in\Gamma(\wedge^{2}f)$. Therefore, 
\[ \exp(t\tau)p(t) = p(t) + t\tau p(t) = p_{0} + \sum\limits_{k=1}^{d}t^{k}(p_{k} + \tau p_{k-1}) + t^{d+1}\tau p_{d}\]
is a polynomial of degree at most $d$ and the coefficient of $t^{d}$ is $p_{d} + \tau p_{d-1}$. 
\end{proof}

\begin{remark}
For polynomial conserved quantities of $\Omega_{0}$-surfaces, \ref{item:polytop} and~\ref{item:polygauge} of Proposition~\ref{prop:polymid} do not necessarily hold. We shall not consider general polynomial conserved quantities of $\Omega_{0}$-surfaces, however in Subsection~\ref{subsec:tublin} we shall consider constant conserved quantities.  
\end{remark}

\begin{corollary}
\label{cor:degreedrop}
Suppose that $p$ is a polynomial conserved quantity of degree $d$ of the middle pencil of $f$. Then for $\tau\in\Gamma(\wedge^{2}f)$, $\tilde{p}(t) :=\exp(t\tau)p(t)$ has degree strictly less than $d$ if and only if $p_{d-1}\in\Gamma (s^{+})^{\perp}$ (or $p_{d-1}\in\Gamma (s^{-})^{\perp}$) and $\tau=\tau^{+}$ (respectively, $\tau=\tau^{-}$).  
\end{corollary}
\begin{proof}
From~\ref{item:polygauge} of Proposition~\ref{prop:polymid}, we have that the coefficient of $t^{d}$ of $\tilde{p}$ is given by $p_{d} +\tau p_{d-1}$. We may write $\tau = \beta \sigma^{+}\wedge \sigma^{-}$, where $\beta$ is a smooth function and $\sigma^{\pm}$ are Christoffel dual lifts. Then by~\ref{item:polytop} of Proposition~\ref{prop:polymid}, we have that 
\[p_{d} +\tau p_{d-1} =  -\tfrac{1}{2}((\sigma^{-},p_{d-1})\sigma^{+} + (\sigma^{+},p_{d-1})\sigma^{-}) +\beta ( (\sigma^{+},p_{d-1})\sigma^{-} - (\sigma^{-},p_{d-1})\sigma^{+}) .\]
Therefore, the $t^{d}$ coefficient of $\tilde{p}$ vanishes if and only if 
\begin{equation}
\label{eqn:dthcoeff}
(\beta -\tfrac{1}{2})(\sigma^{+},p_{d-1}) = 0 = (\beta + \tfrac{1}{2})(\sigma^{-},p_{d-1}).
\end{equation}
Since the top term of $p$ is given by $(\sigma^{+}\odot\sigma^{-})p_{d-1}$, we cannot have that $(\sigma^{+},p_{d-1})$ and $(\sigma^{-},p_{d-1})$ both vanish as this would imply that $p$ has degree strictly less than $d$. Therefore, without loss of generality, assume that $(\sigma^{-},p_{d-1})\neq 0$. Then~(\ref{eqn:dthcoeff}) is equivalent to $\beta =- \frac{1}{2}$ and $(\sigma^{+},p_{d-1})=0$, i.e., $\tau = -\frac{1}{2}\sigma^{+}\wedge \sigma^{-} = \tau^{+}$ and $p_{d-1}\in\Gamma (s^{+})^{\perp}$.
\end{proof}

\begin{corollary}
\label{cor:polynormdeg}
Suppose that $p$ is a polynomial conserved quantity of $d+t\eta^{mid}$. Then the degree $d$ of $p$ is invariant under gauge transformation if and only if $(p(t),p(t))$ is a polynomial of degree $2d-1$. 
\end{corollary}
\begin{proof}
By~\ref{item:polytop} of Proposition~\ref{prop:polymid} we have that $p_{d}\in\Gamma f$. Therefore there is no $2d$-term of $(p(t),p(t))$. Now the coefficient of $t^{2d-1}$ in $(p(t),p(t))$ is $2(p_{d},p_{d-1})$ and by~\ref{item:polytop} of Proposition~\ref{prop:polymid}, 
\[ (p_{d},p_{d-1}) = -(\sigma^{+},p_{d-1})(\sigma^{-},p_{d-1}).\]
Therefore, by Corollary~\ref{cor:degreedrop}, the coefficient of $t^{2d-1}$ vanishes if and only if there exists $\tau\in \Gamma(\wedge^{2}f)$ such that $\exp(t\tau)p(t)$ has degree strictly less than $d$. 
\end{proof}
 
Analogously to~\cite{BS2012,S2008}, we make the following definition:

\begin{definition}
An umbilic-free $\Omega$-surface is a special $\Omega$-surface of type $d$ if the middle pencil of $f$ admits a non-zero polynomial conserved quantity of degree $d$. 
\end{definition}

One should note that a special $\Omega$-surface of type $d$ is automatically a special $\Omega$-surface of type $d+n$, for all $n\in \mathbb{N}$, because one may always multiply $p(t)$ by a real valued polynomial of degree $n$, for example, $t^{n}$. On the other hand a special $\Omega$-surface of type $d$ can also be a special $\Omega$-surface of lower type. 

Note also that type zero special $\Omega$-surfaces do not exist as this would imply that there exists $\mathfrak{q}\in(\mathbb{R}^{4,2})^{\times}$ such that $\mathfrak{q}\in\Gamma f$, implying that $f$ is totally umbilic. 

Now suppose that $f$ is a special $\Omega$-surface of type $d$ with degree $d$ conserved quantity $p$. Let $m$ be a non-zero root of the polynomial $(p(t),p(t))$. Then $p(m)$ is lightlike and is a parallel section of $d+m\eta^{mid}$. If we let $s_{0}:= f\cap \langle p(m)\rangle^{\perp}$ and define 
$\hat{f}:= s_{0}\oplus  \langle p(m)\rangle$, then $\hat{f}$ is a Darboux transform of $f$ with parameter $m$. Unsurprisingly,~\cite{BS2012,S2008} lead us to make the following definition:

\begin{definition}
The Darboux transforms $\hat{f}$ of $f$ such that $p(m)\in \Gamma\hat{f}$ for some $m\in\mathbb{R}^{\times}$ are called the complementary surfaces of $f$ with respect to $p$. 
\end{definition}

Since the degree of $(p(t),p(t))$ is less than or equal to $2d-1$, we have at most $2d-1$ complementary surfaces.

\section{Transformations of polynomial conserved quantities}
We would now like to investigate how polynomial conserved quantities behave when we apply the transformations of Subsection~\ref{subsec:trafos}. Suppose that $f$ is a special $\Omega$-surface of type $d$ and let $p$ be the associated degree $d$ polynomial conserved quantity of the middle pencil of $f$. 

\subsection{Calapso transformations}
Suppose that $f^{t}:=T(t)f$ is a Calapso transform of $f$, where $T(t)$ denotes the local trivialising orthogonal gauge transformations of $d+t\eta^{mid}$.  We now have a result analogous to~\cite[Theorem 3.12]{BS2012}:

\begin{proposition}
\label{prop:calpcq}
The middle pencil of $f^{t}$ admits a degree $d$ polynomial conserved quantity $p^{t}$ defined by
\[p^{t}(s):= T(t)p(s+t)\]
with constant term $p^{t}(0)=T(t)p(t)$. 
\end{proposition}
\begin{proof}
By Proposition~\ref{prop:calgauge}, the middle pencil of $f^{t}$ is given by 
\[ d+s(\eta^{t})^{mid}= T(t)\cdot (d+(s+t)\eta^{mid}).\]
Then it follows immediately that $p^{t}$ is a polynomial conserved quantity of $d+s(\eta^{t})^{mid}$. Furthermore, the coefficient of $s^{d}$ in $p^{t}(s)$ is $T(t)p_{d}\neq 0$. Hence $p^{t}$ has degree $d$. 
\end{proof}

We have thus proved the following Theorem: 

\begin{theorem}
\label{thm:calpcq}
The Calapso transforms of special $\Omega$-surfaces of type $d$ are special $\Omega$-surfaces of type $d$. 
\end{theorem}

\subsection{Darboux transformations}
Suppose that $\hat{f}$ and $f$ are umbilic-free Darboux transforms of each other with parameter $m\in\mathbb{R}^{\times}$. Then by Proposition~\ref{prop:darbmid}, the middle pencil of $\hat{f}$ is given by
\[ d+t\hat{\eta}^{mid} = \Gamma^{\hat{s}}_{s}(1-t/m)\cdot(d+t\eta^{mid})\]
where $\hat{s}\le \hat{f}$ and $s\le f$ are the parallel subbundles of $d+m\eta^{mid}$ and $d+m\hat{\eta}^{mid}$, respectively, implementing these Darboux transforms. Therefore, $\Gamma^{\hat{s}}_{s}(1-t/m)p(t)$ is a conserved quantity of $d+t\hat{\eta}^{mid}$. Using the splitting 
\[ \underline{\mathbb{R}}^{4,2} = s\oplus \hat{s} \oplus (s\oplus \hat{s})^{\perp},\]
we shall write $p(t)$ as 
\[ p(t) = [p(t)]_{s} + [p(t)]_{\hat{s}} + [p(t)]_{(s\oplus\hat{s})^{\perp}}.\]
Thus 
\[ \Gamma^{\hat{s}}_{s}(1-t/m)p(t) = \tfrac{m}{m-t}[p(t)]_{s} + \tfrac{m-t}{m}[p(t)]_{\hat{s}} +  [p(t)]_{(s\oplus\hat{s})^{\perp}}.\]
We then have the following proposition:

\begin{proposition}
\label{prop:darbpoly}
$\hat{p}(t):=(1-t/m)\Gamma^{\hat{s}}_{s}(1-t/m)p(t)$ defines a degree $d+1$ polynomial conserved quantity of $d+t\hat{\eta}^{mid}$. Furthermore, if $p(m)\in \Gamma\hat{s}^{\perp}$ then $\hat{p}(t):= \Gamma^{\hat{s}}_{s}(1-t/m)p(t)$ is a degree $d$ polynomial conserved quantity with $(\hat{p}(t),\hat{p}(t))=(p(t),p(t))$. In either case $\hat{p}(0) = p(0)$. 
\end{proposition}
\begin{proof}
First note that by Proposition~\ref{prop:polymid}, the top term $p_{d}$ of $p(t)$ lies in $f$. Therefore, $[p(t)]_{\hat{s}}$ has degree strictly less than $d$. Hence, 
\[ (1-t/m) \Gamma^{\hat{s}}_{s}(1-t/m)p(t) = [p(t)]_{s} + t\frac{(m-t)^{2}}{m^{2}}[p(t)]_{\hat{s}} + \frac{m-t}{m}[p(t)]_{(s\oplus\hat{s})^{\perp}}\]
is a polynomial conserved quantity of degree $d+1$ of $d+t\hat{\eta}^{mid}$. 

Now let $\sigma\in\Gamma s$ and $\hat{\sigma}\in\Gamma \hat{s}$ such that $(\sigma,\hat{\sigma})=-1$. Then $[p(t)]_{s} =- (p(t),\hat{\sigma})\sigma$. Therefore, if $p(m)\in\Gamma \hat{s}^{\perp}$, then $[p(t)]_{s}$ has a root at $m$ and $\tfrac{m}{m-t}[p(t)]_{s}$ is a polynomial of degree less than $d$. Therefore, $\hat{p}(t)=\Gamma^{\hat{s}}_{s}(1-t/m)p(t)$ is a degree $d$ polynomial conserved quantity of $d+t\hat{\eta}^{mid}$. Furthermore, since $\Gamma_{s}^{\hat{s}}(1-t/m)$ takes values in $\textrm{O}(4,2)$ for all $t$, $(\hat{p}(t),\hat{p}(t))=(p(t),p(t))$.

Finally, we have that in either case $\hat{p}(0)= p(0)$ because $\Gamma^{\hat{s}}_{s}(1)$ is the identity. 
\end{proof}

\begin{corollary}
An umbilic-free Darboux transform $\hat{f}$ of a special $\Omega$-surface $f$ of type $d$ is a special $\Omega$-surface of type $d+1$. Furthermore, if $p(m)\in\Gamma \hat{s}^{\perp}$, where $\hat{s}\le \hat{f}$ is the parallel subbundle of $d+m\eta^{mid}$ implementing this Darboux transform, then $\hat{f}$ is a special $\Omega$-surface of type $d$. 
\end{corollary}

Since $\{d^{t}=d+t\eta^{mid}\}_{t\in\mathbb{R}}$ is a family of metric connections, we have that
\[ d(p(m),\hat{\sigma}) = (d^{m}p(m), \hat{\sigma}) + (p(m), d^{m}\hat{\sigma}) =0,\]
where $\hat{\sigma}\in\Gamma \hat{s}$ is a parallel section of $d^{m}$. Therefore, if $p(m)\in\Gamma \hat{s}^{\perp}$ at a point $p\in\Sigma$, then $p(m)\in\Gamma \hat{s}^{\perp}$ throughout $\Sigma$. By Lemma~\ref{lem:darbcons}, one then deduces that there is a 3-parameter family of Darboux transforms with parameter $m$ satisfying $p(m)\in\Gamma \hat{s}^{\perp}$. Since this holds for every $m\in\mathbb{R}$, we have the following theorem:

\begin{theorem}
\label{thm:darbtyped}
Darboux transforms of special $\Omega$-surfaces of type $d$ are special $\Omega$-surfaces of type $d+1$. Furthermore, there is a 4-parameter family of these Darboux transforms that are special $\Omega$-surfaces of type $d$. 
\end{theorem}

The following proposition provides sufficient conditions for a Ribaucour pair of special $\Omega$-surfaces of type $1$ to determine a Darboux pair. This result follows a similar line of argument as that used in~\cite{BCwip} for the Eisenhart transformation of M\"obius flat surfaces. 

\begin{proposition}
\label{prop:ribtodarb}
Suppose that $f$ and $\hat{f}$ are a Ribaucour pair of special $\Omega$-surfaces of type $1$ whose associated quadratic differentials coincide, i.e., $q=\hat{q}$. Furthermore assume that $p(0)=\hat{p}(0)$, $(p(t),p(t))=(\hat{p}(t),\hat{p}(t))$ and 
\[ \Lambda:= f\cap \langle p(0)\rangle^{\perp} \quad\text{and}\quad \hat{\Lambda}:=\hat{f}\cap \langle p(0)\rangle^{\perp}\]
are immersions with $\Lambda\cap\hat{\Lambda}=\{0\}$. Then $f$ and $\hat{f}$ are either Lie sphere transformations of each other or are Darboux transforms of each other with $f$ and $\hat{f}$ belonging to the respective 4-parameter families detailed in Theorem~\ref{thm:darbtyped}. 
\end{proposition}
\begin{proof}
Let $s_{0}:=f\cap \hat{f}$. Since $f$ and $\hat{f}$ are a Ribaucour pair we may choose lifts (see~\cite[Corollary 2.11]{P2016}) $\sigma\in \Gamma \Lambda$ and $\hat{\sigma}\in \Gamma \hat{\Lambda}$ such that $d\sigma,d\hat{\sigma}\in \Omega^{1}((\Lambda \oplus \hat{\Lambda})^{\perp})$ with $(\sigma, \hat{\sigma})=-1$. Since we also have that $d\sigma,d\hat{\sigma}\in \Omega^{1}((s_{0}\oplus \langle p_{0}\rangle)^{\perp})$, we may write $d\hat{\sigma}=d\sigma\circ R$ for $R\in \Gamma End(T\Sigma)$, whose eigenbundles are the curvature subbundles. 

Now since $p(0) = \hat{p}(0)$ and $(p(t),p(t))=(\hat{p}(t),\hat{p}(t))$, one has that 
\[ p(t) = p_{0} + t(\xi\sigma_{0} + \lambda \sigma) \quad \text{and}\quad  \hat{p}(t) = p_{0} + t(\xi\sigma_{0} + \mu \hat{\sigma}),\]
for $\sigma_{0}\in \Gamma s_{0}$, such that $(\sigma_{0},p_{0})=-1$, $\xi\in\mathbb{R}$ and smooth functions $\lambda$ and $\mu$. Let $\eta':= \eta - d\tau$ and $\hat{\eta}':= \hat{\eta} - d\hat{\tau}$, where $\tau :=\lambda\sigma_{0}\wedge \sigma$ and $\hat{\tau} := \mu\sigma_{0}\wedge \hat{\sigma}$. Then the linear conserved quantities, $p'$ and $\hat{p}'$, of $d+t\eta'$ and $d+t\hat{\eta}'$, respectively, satisfy 
\[ p'(t) := \exp(t\tau)p(t) = p_{0}+t\,\xi\sigma_{0} = \exp(t\hat{\tau})\hat{p}(t) =: \hat{p}'(t).\]
The condition that $p'$ and $\hat{p}'$ are linear conserved quantities implies that
\begin{equation}
\label{eqn:ahata}
\eta' = \sigma \wedge d\hat{\sigma}\circ A +\xi\, \sigma_{0}\wedge d\sigma_{0} \quad \text{and}\quad  \hat{\eta}' = \hat{\sigma}\wedge d\sigma\circ \hat{A} +\xi\, \sigma_{0}\wedge d\sigma_{0} ,
\end{equation}
for $A,\hat{A}\in \Gamma End(T\Sigma)$, whose eigenbundles are the curvature subbundles. The condition that $q=\hat{q}$ implies that $A=\hat{A}$. Then since 
\[ \eta' - \hat{\eta}' =  \sigma \wedge d\hat{\sigma}\circ A - \hat{\sigma}\wedge d\sigma\circ A\]
is closed one has that $d\sigma\circ A$ and $d\hat{\sigma}\circ A$ are closed and 
\begin{equation*}
d\sigma\curlywedge d\hat{\sigma}\circ A  - d\hat{\sigma}\curlywedge d\sigma\circ A=0.
\end{equation*}
This is equivalent to 
\[0= r_{1}\alpha_{1} +r_{2}\alpha_{2} - r_{1}\alpha_{2} -r_{2}\alpha_{1} = (r_{1}-r_{2})(\alpha_{1}-\alpha_{2}),\]
where $r_{1},r_{2}$ are the eigenvalues of $R$ and $\alpha_{1},\alpha_{2}$ are the eigenvalues of $A$. If $r_{1}-r_{2}=0$ then one has that $r_{1}$ is constant, because $d\hat{\sigma}$ is closed, and thus $w:=\langle \hat{\sigma}-r_{1}\sigma\rangle$ is constant and $\hat{f}$ is obtained from $f$ by reflecting $f$ across $w$. Hence $\hat{f}$ is a Lie sphere transformation of $f$. 

If $r_{1}-r_{2}\neq 0$, then we have that $A=\alpha_{1}\, id$. The closure of $d\sigma\circ A$ and the fact that $\Lambda$ is an immersion then implies that $\alpha_{1}$ is constant. $\alpha_{1}$ cannot be zero because then one has from~(\ref{eqn:ahata}) that $\sigma_{0}\wedge d\sigma_{0}$ is closed, implying that $\sigma_{0}$ does not immerse. $q$ and $\hat{q}$ would then be degenerate quadratic differentials, contradicting that $f$ and $\hat{f}$ are $\Omega$-surfaces. One then deduces from~(\ref{eqn:ahata}) that $(d+m\eta)\hat{\sigma}=0$ and $(d+m\hat{\eta})\sigma=0$, where $m:=\alpha_{1}^{-1}$. Hence $f$ and $\hat{f}$ are Darboux transforms of each other with parameter $m$. Furthermore, we have that $p'(m) = \hat{p}'(m) = p_{0}+m\xi\sigma_{0}$ and thus $\Lambda, \hat{\Lambda}\perp p'(m)$. Thus, $f$ (respectively, $\hat{f}$) belong to the 4-parameter family of Darboux transforms of $\hat{f}$ (respectively, $f$) in Theorem~\ref{thm:darbtyped}. 
\end{proof}

\section{Type 1 special-$\Omega$ surfaces}
In this section we shall see that special $\Omega$-surfaces of type 1, i.e., $\Omega$-surfaces whose middle pencil admits a linear conserved quantity $p(t)$, include isothermic surfaces, Guichard surfaces and $L$-isothermic surfaces. Furthermore the familiar transformations of these surfaces are restrictions of the transformations of Subsection~\ref{subsec:trafos}. For example the Eisenhart transformations for Guichard surfaces are Darboux transformations preserving the linear conserved quantity. 

Suppose that $f$ is a special $\Omega$-surface of type 1 and let $p(t) = p_{0} + tp_{1}$ be the associated linear conserved quantity of the middle pencil of $f$. By Proposition~\ref{prop:polymid}, $p_{0}$ is constant and $p_{1}\in\Gamma f$. We may also deduce the following lemma:

\begin{lemma}
\label{lem:lincon}
Suppose that $p_{0}\in \mathbb{R}^{4,2}$. Then $(\sigma^{\pm},p_{0})$ are constant if and only if $p(t) = \exp(-t\, \sigma^{+}\odot \sigma^{-})p_{0}$ is a linear conserved quantity of the middle pencil.
\end{lemma}
\begin{proof}
The necessity of this lemma follows immediately from part~\ref{item:polytop} of Proposition~\ref{prop:polymid}. One can quickly deduce the sufficiency by using the form of the middle potential given in~(\ref{eqn:midchris}).
\end{proof}

Using Lemma~\ref{lem:totumb} we obtain the following corollary:
\begin{corollary}
\label{cor:legperp}
$f$ nowhere lies in $\langle p_{0}\rangle^{\perp}$. 
\end{corollary}
\begin{proof}
Suppose that at a point $p\in\Sigma$, $f(p)\le \langle p_{0}\rangle^{\perp}$. Then $(\sigma^{\pm}(p),p_{0})=0$ and by Lemma~\ref{lem:lincon}, $(\sigma^{\pm},p_{0})=0$ throughout $\Sigma$. Therefore, since $\sigma^{\pm}$ span $f$, $f\le \langle p_{0}\rangle^{\perp}$. Then by Lemma~\ref{lem:totumb}, this contradicts $f$ being an umbilic-free Legendre map. 
\end{proof}

\subsection{Isothermic surfaces}
\label{subsec:isosurf}
Suppose that $\mathfrak{p}\in\mathbb{R}^{4,2}$ is a point sphere complex. Then $\langle \mathfrak{p}\rangle^{\perp}$ is a (Riemannian or Lorentzian) conformal subgeometry of $\mathbb{R}^{4,2}$. Let $\mathcal{L}^{\mathfrak{p}}$ denote the lightcone of $\langle\mathfrak{p}\rangle^{\perp}$. In~\cite{BDPP2011, BS2012, H2003, S2008}, isothermic surfaces are characterised as the surfaces $\Lambda:\Sigma\to\mathbb{P}(\mathcal{L}^{\mathfrak{p}})$ that admit a non-zero closed 1-form 
\[ \eta\in  \Omega^{1}(\Lambda\wedge \Lambda^{(1)}).\]
Let $f:\Sigma\to \mathcal{Z}$ be the Legendre lift of $\Lambda$. Then $\Lambda = f\cap\langle \mathfrak{p}\rangle^{\perp}$ and $\eta$ takes values in $f\wedge f^{\perp}$. Furthermore, the quadratic differential 
\[ q(X,Y) = tr(\sigma \to \eta(X)d_{Y}\sigma)\]
coincides with the holomorphic\footnote{That is, locally there exists a holomorphic coordinate $z$ on $\Sigma$ such that $q^{2,0} := q(\frac{\partial}{\partial z},\frac{\partial}{\partial z})dz^{2} = dz^{2}$ and $I = e^{2u}dz d\bar{z}$.} (with respect to the conformal structure induced by $\Lambda$) quadratic differential defined in~\cite{BS2012,S2008}. Thus, $q$ is non-degenerate and $f$ is an $\Omega$-surface. Furthermore, 
\[ (d+t\eta)\mathfrak{p} =0,\]
i.e., $\mathfrak{p}$ is a constant conserved quantity of $d+t\eta$. Thus, if $\tau\in\Gamma(\wedge^{2}f)$ such that the middle pencil of $f$ is given by
\[ d+t\eta^{mid} = \exp(t\tau)\cdot (d+t\eta),\]
then we have that $p(t) = \exp(t\tau)\mathfrak{p}$ is a linear conserved quantity of $d+t\eta^{mid}$. Moreover, 
$(p(t),p(t))= (\mathfrak{p},\mathfrak{p})$ is a non-zero constant. 

Conversely, suppose that $f$ is a special $\Omega$-surface of type 1 with linear conserved quantity $p$ and suppose that $(p(t),p(t))$ is a non-zero constant. If we let $\mathfrak{p}:= p(0)$, then $\mathfrak{p}$ is a point sphere complex and $\langle\mathfrak{p}\rangle^{\perp}$ defines a (Riemannian or Lorentzian) conformal geometry. By Corollaries~\ref{cor:degreedrop} and~\ref{cor:polynormdeg}, we have that one of the isothermic sphere congruences, without loss of generality $\Lambda:=s^{+}$, of $f$ takes values in $\langle\mathfrak{p}\rangle^{\perp}$. Then $\Lambda$ is an isothermic surface and 
\[ \eta^{+} \in \Omega^{1}(\Lambda\wedge \Lambda^{(1)})\]
is its associated closed 1-form. We have therefore arrived at the following theorem:

\begin{theorem}
\label{thm:omegaiso}
Special $\Omega$-surfaces of type 1 whose degree 1 polynomial conserved quantity $p$ satisfies $(p(t),p(t))$ being a non-zero constant are the isothermic surfaces of the conformal geometry defined by $\langle p(0)\rangle^{\perp}$. 
\end{theorem}

We shall now see how the classical transformations of isothermic surfaces are induced by the transformations of Subsection~\ref{subsec:trafos}: suppose that $f$ is an umbilic-free $\Omega$-surface such that $\Lambda:=s^{+}$ is an isothermic surface in $\langle \mathfrak{p}\rangle^{\perp}$. Then
\[ p(t):=\exp(t\tau)\mathfrak{p}\]
is a polynomial conserved quantity of the middle pencil, where $\tau = \frac{1}{2}\sigma^{+}\wedge\sigma^{-}$ for Christoffel dual lifts $\sigma^{\pm}$.

\subsubsection{Calapso transforms}
Suppose that $f^{t}  =T(t)f$ is a Calapso transform of $f$. Since the Calapso transforms are well defined on gauge orbits (see, Section~\ref{subsubsec:caltrafos}), we may assume that $T(t)$ is the gauge transformation of $d+t\eta^{+}$. Now 
\[ d(T(t)\mathfrak{p}) = T(t)(d+t\eta^{+})\mathfrak{p} = 0.\]
Thus, $T(t)\mathfrak{p}$ is constant and, by premultiplying by an appropriate Lie sphere transformation, we may assume that it is $\mathfrak{p}$. Then $T(t)\Lambda\le f^{t}$ is a Calapso transform of the isothermic surface $\Lambda$ in the sense of \cite{BS2012,H2003, S2008}.

\begin{theorem}
The Calapso transforms of $f$ are the Legendre lifts of the Calapso transforms of $\Lambda$. 
\end{theorem} 

\subsubsection{Darboux transforms}
Suppose that $\hat{f}$ is an umbilic-free Darboux transform of $f$ with parameter $m\in\mathbb{R}^{\times}$. Assume that a parallel section $\hat{\sigma}\le \hat{f}$ of $d+m\eta^{mid}$ satisfies $\hat{\sigma}\in\Gamma \langle p(m)\rangle^{\perp}$. Then $\hat{\sigma}^{+}:=\exp(m\tau)\hat{\sigma}$ is a parallel section of $d+m\eta^{+}$ and 
\[  (\hat{\sigma}^{+},\mathfrak{p}) = ( \exp(m\tau)\hat{\sigma}, \exp(m\tau)p(m)) = (\hat{\sigma},p(m))=0.\]
Thus, $\hat{s}^{+}:=\langle \hat{\sigma}^{+}\rangle$ is a Darboux transform in the sense of \cite{H2003, S2008, BS2012, BDPP2011} of the isothermic surface $s^{+}$. 

Conversely, if $\hat{\Lambda}$ is a Darboux transform of $\Lambda$ with parameter $m$ then $\hat{s}:= \exp(-m\tau)\hat{\Lambda}$ is a parallel subbundle of $d+m\eta^{mid}$ and $\hat{s}\le \langle p(m)\rangle^{\perp}$, since $p(m) = \exp(-m\tau)\mathfrak{p}$. 
We have therefore arrived at the following theorem:

\begin{theorem}
\label{thm:darbisosurf}
The Darboux transforms of an isothermic surface constitute a 4-parameter family of Darboux transforms of its Legendre lift. 
\end{theorem}

\subsubsection{The Christoffel transformation}
Now suppose that $\mathfrak{p}$ satisfies $|\mathfrak{p}|^{2}=-1$ and let $\mathfrak{q}_{\infty}\in\langle\mathfrak{p}\rangle^{\perp}$ be a null space form vector. Then $\mathfrak{Q}^{3}$ is isometric to a Euclidean geometry. Let $\mathfrak{f}:\Sigma\to \mathfrak{Q}^{3}$ denote the corresponding space form projection of $f$ and let $x:\Sigma\to\mathbb{R}^{3}$ be a corresponding surface in Euclidean space. Then 
\[ \eta^{+} = \mathfrak{f}\wedge d\mathfrak{f}\circ A\] 
for some $A\in\Gamma End(T\Sigma)$ and $(\eta^{+}\mathfrak{p},\mathfrak{q}_{\infty})=0$. By comparing this with Subsection~\ref{subsec:asssurf}, we have that there is an associate surface $x^{D}$ of $x$ such that 
\[ \frac{1}{\kappa_{1}\kappa_{2}^{D}} + \frac{1}{\kappa_{2}\kappa_{1}^{D}} =0.\]
One can then deduce that the conformal structures induced by $x$ and $x^{D}$ are equivalent. Therefore, since $x$ and $x^{D}$ have parallel curvature directions and induce the same conformal structure, they are Christoffel transforms of each other. 

\subsection{Guichard surfaces}
In this subsection we shall characterise Guichard surfaces in conformal geometries amongst special $\Omega$-surfaces of type 1. We will then see how the well known transformations of these surfaces are induced by the transformations of the underlying $\Omega$-surface. This exposition has been partly outlined in~\cite[Section 7.5.2]{C2012i}. 

Suppose that $\mathfrak{p}\in \mathbb{R}^{4,2}$ is a point sphere complex for a conformal geometry $\langle \mathfrak{p}\rangle^{\perp}$. Let $\mathcal{L}^{\mathfrak{p}}$ denote the lightcone of $\langle \mathfrak{p}\rangle^{\perp}$. Recall from Section~\ref{subsubsec:conf} that spheres in this conformal geometry are represented by $(3,1)$-planes $V\le \langle \mathfrak{p}\rangle^{\perp}$. Therefore, given a two-dimensional manifold $\Sigma$, one can represent a sphere congruence as a rank 4 subbundle $V$ of the bundle $\Sigma\times \langle \mathfrak{p}\rangle^{\perp}$ with induced signature $(3,1)$. One may then split the trivial connection $d$ on $\underline{\mathbb{R}}^{4,2}=V\oplus V^{\perp}$ as
\[ d = \mathcal{D}^{V}+\mathcal{N}^{V},\]
where $\mathcal{D}^{V}$ is the sum of the induced connections on $V$ and $V^{\perp}$ and $\mathcal{N}^{V}\in \Omega^{1}(V\wedge V^{\perp})$. Let $s$ and $\tilde{s}$ denote the null rank $1$ subbundles of $V^{\perp}$. We may then write 
\[ \mathcal{N}^{V} = \mathcal{N}^{s} + \mathcal{N}^{\tilde{s}}\]
where $\mathcal{N}^{s}\in \Omega^{1}(\tilde{s}\wedge V)$ and $\mathcal{N}^{\tilde{s}}\in \Omega^{1}(s\wedge V)$. We say that a map $\Lambda:\Sigma\to \mathbb{P}(\mathcal{L}^{\mathfrak{p}})$ \textit{envelops} $V$ if $\Lambda^{(1)}\subset V$, equivalently, $\mathcal{N}^{V}\Lambda=0$. One then has that $f:=\Lambda\oplus s$ and $\tilde{f}:=\Lambda \oplus \tilde{s}$ are the Legendre maps enveloping $\Lambda$ with opposite orientations. 

In~\cite{BCwip}, M\"obius flat submanifolds are derived and studied. In codimension 1, these coincide with Guichard surfaces:

\begin{definition}[{\cite{BCwip}}]
\label{def:mflat}
$\Lambda$ is a Guichard surface if, for some (and in fact, any) enveloped sphere congruence $V$, there exists $\chi^{V}\in \Omega^{1}(\Lambda\wedge \Lambda^{(1)})$ such that
\[ d^{V}_{t}:= \mathcal{D}^{V} + t\mathcal{N}^{V}+ (t^{2}-1)\chi^{V},\]
is flat for all $t\in \mathbb{R}$. 
\end{definition}

Associated to a Guichard surface is a quadratic differential $q^{\Lambda}\in \Gamma S^{2}(T\Sigma)^{*}$ defined by
\[ q^{\Lambda}(X,Y) = 2\text{tr}(\Lambda\to \Lambda, \sigma\mapsto \chi^{V}_{X} \mathcal{D}_{Y}^{V}\sigma) + \text{tr}(\mathcal{N}_{X}^{V}\circ \mathcal{N}_{Y}^{V}|_{V^{\perp}}). \]
It is shown in~\cite{BCwip} that this quadratic differential is independent of the choice of $V$. Furthermore, if $q^{\Lambda}$ is a degenerate quadratic differential, then $\Lambda$ is a channel surface. 

So now let us suppose that $\Lambda$ is a Guichard surface with non-degenerate $q^{\Lambda}$, i.e., $\Lambda$ is a non-channel Guichard surface. Let $V$ be an enveloped sphere congruence of $\Lambda$. Consider now
\[ \Gamma^{s}_{\tilde{s}}(u) = \left\{\begin{array}{ll} u & \text{on $s$,}\\
id & \text{on $V^{\perp}$,}\\ 
u^{-1} & \text{on $\tilde{s}$.}\end{array}
\right. \]
Then 
\begin{align*}
\Gamma^{s}_{\tilde{s}}(t)\cdot d^{V}_{t} &= \Gamma^{s}_{\tilde{s}}(t)\cdot (\mathcal{D}^{V} + t\mathcal{N}^{s} + t \mathcal{N}^{\tilde{s}} + (t^{2}-1)\chi^{V})\\
&= \mathcal{D}^{V} + \mathcal{N}^{s} + t^{2}\mathcal{N}^{\tilde{s}} + (t^{2}-1)\chi^{V}\\
&=  \mathcal{D}^{V} + \mathcal{N}^{s} + \mathcal{N}^{\tilde{s}} + (t^{2}-1) (\chi^{V}+ \mathcal{N}^{\tilde{s}})\\
&= d+\tfrac{t^{2}-1}{2}\eta,
\end{align*}
where $\eta:=2(\chi^{V}+ \mathcal{N}^{\tilde{s}})\in \Omega^{1}(f\wedge f^{\perp})$. Since $d^{V}_{t}$ is flat, one has that $d+ \frac{t^{2}-1}{2}\eta$ is flat and thus $\eta$ is closed. Furthermore, one deduces that the quadratic differential $q$ of $\eta$, i.e., 
\[ q(X,Y) = \text{tr}(\sigma\mapsto \eta_{X}d_{Y}\sigma),\]
coincides with $q^{\Lambda}$. Thus, $q$ is non-degenerate and $f$ is an $\Omega$-surface. Now $\mathfrak{p}$ is a constant conserved quantity of $d^{V}_{t}$, thus $t\Gamma^{s}_{\tilde{s}}(t)\mathfrak{p}$ is a conserved quantity of $d+\frac{t^{2}-1}{2}\eta$. Moreover, one may write
\[ t\Gamma^{s}_{\tilde{s}}(t)\mathfrak{p} = \mathfrak{p}+\tfrac{t^{2}-1}{2}\sigma  \]
for some $\sigma\in \Gamma s$. By reparameterising, one has that $p(t):= \mathfrak{p}+t\sigma$ is a linear conserved quantity of $d+t\eta$. Furthermore, $(p(t),p(t))$ is a linear polynomial with non-zero constant term.  

Conversely, suppose that $f$ is an $\Omega$-surface whose middle pencil admits a linear conserved quantity $p$ such that $(p(t),p(t))$ is a linear polynomial with non-zero constant term. Thus, $\mathfrak{p}:=p(0)$ defines a conformal geometry and we denote by $\Lambda:=f\cap \langle \mathfrak{p}\rangle^{\perp}$ the corresponding projection. $V:=\langle p(0), p(1)\rangle^{\perp}$ defines an enveloping sphere congruence of $\Lambda$ and we may write $\eta^{mid}= \eta_{\Lambda}+\eta_{s}$, where $\eta_{\Lambda}\in \Omega^{1}(\Lambda\wedge V)$ and $\eta_{s}\in \Omega^{1}(s\wedge V)$ for $s:=\langle p_{1}\rangle$. Since $(p(t),p(t))$ is a linear polynomial, there exists $t_{0}\in\mathbb{R}$ such that $\tilde{s}:=\langle p(t_{0})\rangle\le V^{\perp}$ is null. Hence, $s$ and $\tilde{s}$ are the null rank 1 subbundles of $V^{\perp}$. Now, 
\[ 0 = (d+t_{0}\eta^{mid}) p(t_{0}) =  \mathcal{N}^{\tilde{s}} p(t_{0}) + t_{0}\eta_{s}p(t_{0}).\]
Hence, $ \mathcal{N}^{\tilde{s}} = -t_{0}\eta_{s}$. Without loss of generality, let us assume that $t_{0}=-\frac{1}{2}$, so that $\eta_{s} = 2 \mathcal{N}^{\tilde{s}}$. Then by defining $\chi^{V}:= \frac{1}{2}\eta_{\Lambda}$, we have that 
\[ \eta^{mid} = 2(\chi^{V} +  \mathcal{N}^{\tilde{s}}).\]
One then deduces that
\[ d^{V}_{t} = \mathcal{D}^{V} + t\mathcal{N}^{V} + (t^{2}-1)\chi^{V}=\Gamma^{s}_{\tilde{s}}(t^{-1})\cdot (d + \tfrac{t^{2}-1}{2}\eta^{mid}).\]
Since $d+\frac{t^{2}-1}{2}\eta^{mid}$ is flat for all $t$, $d^{V}_{t}$ is flat for all $t$. Hence, $\Lambda$ is M\"obius flat. Moreover, $q^{\Lambda}$ coincides with $q$. Thus, we have proved the following theorem: 

\begin{theorem}
\label{thm:guichard}
Special $\Omega$-surfaces of type 1 whose degree 1 polynomial conserved quantity $p$ satisfies $(p(t),p(t))$ being linear with non-zero constant term are the non-channel Guichard surfaces in the conformal geometry of $\langle p(0)\rangle^{\perp}$. 
\end{theorem}

\subsubsection{Calapso transforms}
Let $f^{t}:=T(t)f$ be a Calapso transform of $f$. Then by Proposition~\ref{prop:calpcq}, the middle pencil of $f^{t}$ admits a linear conserved quantity $p^{t}$ defined by 
\[p^{t}(s)=T(t)p(t+s).\]
Now since $T(t)$ take values in $\textrm{O}(4,2)$, we have that
\[(p^{t}(s),p^{t}(s))= (p(t+s),p(t+s)).\]
Therefore, $(p^{t}(s),p^{t}(s))$ is a linear polynomial with constant term $(p(t),p(t))$. Since $(p(t),p(t))$ is a linear polynomial in $t$ with non-zero constant term, it admits a single root which we shall denote $t_{0}$. By applying Theorem~\ref{thm:guichard}, we obtain the following theorem:

\begin{theorem}
If $t\neq t_{0}$, then the Calapso transform $f^{t}$ projects to a Guichard surface in the conformal geometry of $\langle T(t)p(t)\rangle^{\perp}$.  
\end{theorem}

\begin{remark}
\label{rem:guichroot}
The Calapso transform $f^{t_{0}}$ admits a linear polynomial $p^{t_{0}}$ such that $(p^{t_{0}}(s),p^{t_{0}}(s))$ is a linear polynomial with vanishing constant term. Therefore Theorem~\ref{thm:guichard} does not apply in this case. 
\end{remark}

In~\cite{BCwip} a spectral deformation is defined for Guichard surfaces $\Lambda$. Given an enveloping sphere congruence $V$, the 1-parameter family of connections $d^{V}_{t}$ is flat for all $t$. Thus there exist trivialising gauge transformations $\Phi^{V}_{t}$ such that $\Phi^{V}_{t}\cdot d^{V}_{t}=d$. For $r\in \mathbb{R}$ we then say that $\Lambda_{r}:=\Phi^{V}_{r} \Lambda$ is a \textit{T-transform} of $\Lambda$. On the other hand, we have that 
\[ d^{V}_{t}= \Gamma^{s}_{\tilde{s}}(t^{-1})\cdot (d+\tfrac{t^{2}-1}{2}\eta).\]
It follows that $\Phi^{V}_{t} = T(\frac{t^{2}-1}{2})\Gamma^{s}_{\tilde{s}}(t)$. Thus, $\Phi^{V}_{r}\Lambda = T(\frac{r^{2}-1}{2})\Lambda$. Hence the Calapso transforms of the underlying $\Omega$-surface coincide with the T-transforms of $\Lambda$.

\subsubsection{The Eisenhart transformation}
In~\cite{E1914}, Eisenhart determines a B\"acklund type transformation for Guichard surfaces, which has come to be known as the \textit{Eisenhart transformation}. A conformally invariant description of this transformation is given in~\cite{BCwip}, and we will now show how this is induced by certain Darboux transforms of the underlying $\Omega$-surfaces. 

Firstly let us assume that $f$ is a special $\Omega$-surface whose linear conserved quantity $p$ of the middle pencil satisfies $(p(t),p(t))$ being a linear polynomial with non-zero constant term. Let $\Lambda:= f\cap \langle p(0)\rangle^{\perp}$. Then by Theorem~\ref{thm:guichard}, $\Lambda$ is a Guichard surface. 

Suppose that $\hat{\Lambda}$ is an Eisenhart transform of $\Lambda$ with parameter $m$. Then, by~\cite{BCwip}, there exists a sphere congruence $V_{r}$ enveloping $\Lambda$ and $\hat{\Lambda}$ such that $\hat{\Lambda}$ is a parallel subbundle of $d^{V_{r}}_{m}$. Now for the appropriate choice of gauge potential $\eta$, one has that 
\begin{equation}
\label{eqn:eisrel}
d^{V_{r}}_{t} = \Gamma^{s}_{\tilde{s}}(t^{-1})\cdot (d+\tfrac{t^{2}-1}{2}\eta),
\end{equation}
where $s$ and $\tilde{s}$ are the rank 1 null subbundles of $V_{r}^{\perp}$. Since $\hat{\Lambda}$ is a parallel subbundle  $d^{V_{r}}_{m}$, one has that $\Gamma^{s}_{\tilde{s}}(m^{-1})\hat{\Lambda}$ is a parallel subbundle of $d+\tfrac{m^{2}-1}{2}\eta$. Moreover, since $\hat{\Lambda}$ is enveloped by $V_{r}$, $\hat{\Lambda}\le V_{r}$ and thus $\Gamma^{s}_{\tilde{s}}(m^{-1})\hat{\Lambda}=\hat{\Lambda}$. Now by defining $\hat{f}:= s_{0}\oplus \hat{\Lambda}$, where $s_{0}:= f\cap \hat{\Lambda}^{\perp}$, we have that $\hat{f}$ is a Darboux transform of $f$ with parameter $\frac{m^{2}-1}{2}$. Furthermore, 
\[ \hat{\Lambda}\le V_{r} = (\langle p(0)\rangle \oplus s)^{\perp} = \langle \tilde{p}(0), \tilde{p}(1)\rangle^{\perp},\]
where $\tilde{p}$ is the linear conserved quantity of $\eta$. Thus, $\hat{\Lambda}\perp \tilde{p}(\frac{m^{2}-1}{2})$. Now when we let $\hat{s}\le \hat{f}$ denote the parallel subbundle of $d+\frac{m^{2}-1}{2}\eta^{mid}$, this condition is equivalent to $\hat{s}\perp p(\frac{m^{2}-1}{2})$.

Conversely, suppose that $\hat{f}$ is a Darboux transform of $f$ with parameter $\frac{m^{2}-1}{2}$ such that the parallel subbundle $\hat{s}\le \hat{f}$ of $d+\frac{m^{2}-1}{2}\eta^{mid}$ satisfies $\hat{s}\perp p(\frac{m^{2}-1}{2})$. Then by Proposition~\ref{prop:darbfree} there exists a gauge potential $\eta$ in the gauge orbit of $\eta^{mid}$ for which $\hat{\Lambda}=\hat{f}\cap \langle p(0)\rangle^{\perp}$ is a parallel subbundle. Let $\tilde{p}$ be the corresponding linear conserved quantity of $d+t\eta$. Then one has that $\hat{\Lambda} \perp \tilde{p}(\frac{m^{2}-1}{2})$. Let $V:= \langle \tilde{p}(0),\tilde{p}(1)\rangle^{\perp}$. Then $V$ envelopes $\Lambda$ and $\hat{\Lambda}$. Furthermore, $\hat{\Lambda}$ is a parallel subbundle of $d^{V}_{m}$, where
\begin{equation*}
d^{V}_{t} = \Gamma^{s}_{\tilde{s}}(t^{-1})\cdot (d+\tfrac{t^{2}-1}{2}\eta),
\end{equation*}
and $s$ and $\tilde{s}$ are the null rank $1$ subbundles of $V^{\perp}$. Therefore, by~\cite{BCwip}, $\hat{\Lambda}$ is an Eisenhart transform of $\Lambda$ with parameter $m$. We have therefore arrived at the following theorem:

\begin{theorem}
The Eisenhart transforms of a Guichard surface constitute a 4-parameter family of Darboux transforms of its Legendre lift. 
\end{theorem}

\subsubsection{The associate surface}
Let $\mathfrak{p}:=p(0)$ and suppose now that $|\mathfrak{p}|^{2}=-1$. Choose a null space form vector $\mathfrak{q}_{\infty}\in\langle\mathfrak{p}\rangle^{\perp}$. Then $\mathfrak{Q}^{3}$ is isometric to Euclidean $3$-space. As usual, let $\mathfrak{f}:\Sigma\to \mathfrak{Q}^{3}$ denote the space form projection of $f$ into $\mathfrak{Q}^{3}$ and let $\mathfrak{t}:\Sigma\to \mathfrak{P}^{3}$ denote its tangent plane congruence. Now we may choose a 1-form $\tilde{\eta}$ in the gauge orbit of $\eta^{mid}$ such that the linear conserved quantity $\tilde{p}$ of $\tilde{\eta}$ satisfies $\tilde{p}_{1}\perp \mathfrak{q}_{\infty}$. After rescaling by a constant if necessary, one can deduce that $\tilde{\eta}$ has the form 
\[ \tilde{\eta}= \mathfrak{f}\wedge d\mathfrak{f}\circ A -\mathfrak{t}\wedge d\mathfrak{t}\]
for some $A\in \Gamma End(T\Sigma)$. Therefore, by comparing with Subsection~\ref{subsec:asssurf}, any projection $x:\Sigma\to \mathbb{R}^{3}$ of $\mathfrak{f}$ with unit normal $n:\Sigma\to S^{2}$ admits an associate surface $x^{D}$ such that the associate Gauss map is given by the unit normal of $x$, i.e., $\hat{x}=n$. Thus 
\[  0 = \frac{1}{\kappa_{1}\kappa^{D}_{2}} +  \frac{1}{\kappa_{2}\kappa^{D}_{1}} -  \frac{1}{\hat{\kappa}_{1}} - \frac{1}{\hat{\kappa}_{2}}=\frac{1}{\kappa_{1}\kappa^{D}_{2}} +  \frac{1}{\kappa_{2}\kappa^{D}_{1}} + 2.\]
Hence, $x^{D}$ is an associate surface in the sense of Guichard~\cite{G1900}.

\subsection{L-isothermic surfaces}
\label{subsec:lisosurf}
$L$-isothermic surfaces were originally discovered by Blaschke \cite{B1929} and have been the subject of interest recently in for example \cite{E1962,MN1996,MN1997,MN2000,MN2014,RS2009, S2009}. They are the surfaces in $\mathbb{R}^{3}$ that admit curvature line coordinates that are conformal with respect to the third fundamental form of the surface, or as Musso and Nicolodi~\cite{MN2000} put it, there exists a holomorphic\footnote{That is, locally there exists a complex coordinate $z$ on $\Sigma$ such that $q^{2,0} :=q(\frac{\partial}{\partial z},\frac{\partial}{\partial z}) dz^{2}=dz^{2}$ and $I\! I\! I = e^{2u}dz d\bar{z}$.} (with respect to the third fundamental form) quadratic differential $q$ that commutes with the second fundamental form, i.e., if we use the complex structure induced on $\Sigma$ by $I\! I\! I$ to split the second fundamental form into bidegrees,
\[ I\! I = I\! I^{2,0} + I\! I^{1,1} + I\! I^{0,2},\]
then $I\! I^{2,0} = \mu q^{2,0}$, for some real valued function $\mu:\Sigma\to\mathbb{R}$. In~\cite{MN2000}, $L$-isothermic surfaces were also characterised in terms of the standard model for Laguerre geometry $\mathbb{R}^{3,1}$ (see for example~\cite{B1929,C2008}). In this subsection we will show that Legendre lifts of $L$-isothermic surfaces are the special $\Omega$-surfaces of type 1 whose linear conserved quantity $p$ of the middle pencil satisfies $(p(t),p(t))=0$.

Recall from Subsection~\ref{subsec:laguerre} that a non-zero lightlike vector $\mathfrak{q}_{\infty}$ defines a Laguerre geometry and that by choosing $\mathfrak{q}_{0}\in\mathcal{L}$ such that $(\mathfrak{q}_{0},\mathfrak{q}_{\infty})=-1$ and $\mathfrak{p}\in \langle\mathfrak{q}_{0},\mathfrak{q}_{\infty}\rangle^{\perp}$ such that $|\mathfrak{p}|^{2}=-1$, one can show that
\[ \mathfrak{Q}^{3}=\{y\in \mathcal{L}: (y,\mathfrak{q}_{\infty})=-1,(y,\mathfrak{p})=0\}\]
 is isometric to $\mathbb{R}^{3} \cong\langle \mathfrak{q}_{\infty},\mathfrak{q}_{0},\mathfrak{p}\rangle^{\perp}$. 

Now suppose that $f:\Sigma\to \mathcal{Z}$ is a Legendre map and that $f$ projects to a surface $\mathfrak{f}:\Sigma\to\mathfrak{Q}^{3}$ with tangent plane congruence $\mathfrak{t}:\Sigma\to \mathfrak{P}^{3}$. Then 
\[\mathfrak{f} = x+\mathfrak{q}_{0} +\frac{1}{2}(x,x)\mathfrak{q}_{\infty}\quad \text{and}\quad \mathfrak{t} = n+\mathfrak{p}+(n,x)\mathfrak{q}_{\infty},\]
where $x:\Sigma\to\mathbb{R}^{3}$ is the corresponding surface in $\mathbb{R}^{3}$ with unit normal $n:\Sigma\to S^{2}$. Suppose that there exists a holomorphic (with respect to the third fundamental form of $x$, $I\! I\! I = (dn,dn)$) quadratic differential $q$ that commutes with the second fundamental form of $x$, $I\! I = -(dx,dn)$. This implies that if we let $Q\in \Gamma End(T\Sigma)$ such that 
\[ q = (dn,dn\circ Q),\]
then $Q$ is trace-free and symmetric with respect to $I\! I \! I$ and the 2-tensor
\[ (dx,dn\circ Q) \]
is symmetric. Now let
\begin{equation*}
\eta:= \mathfrak{t}\wedge d\mathfrak{t}\circ Q
= (n+\mathfrak{p}+(n,x)\mathfrak{q}_{\infty})\wedge (dn\circ Q + (dn\circ Q,x)\mathfrak{q}_{\infty}).
\end{equation*}
Then $d\eta$ is equal to
\begin{equation*}
dn\curlywedge dn\circ Q +  \mathfrak{q}_{\infty}\wedge((dn\curlywedge dn\circ Q)x) 
+  (n+\mathfrak{p}+(n,x)\mathfrak{q}_{\infty})\wedge ( d(dn\circ Q) + d(dn\circ Q, x)\mathfrak{q}_{\infty}).\end{equation*}
It follows from the fact that $Q$ is trace-free that $dn\curlywedge dn\circ Q = 0$. Furthermore, one can check that $q$ being holomorphic implies that $dn\circ Q$ is closed. Finally, for any $X,Y\in \Gamma T\Sigma$, 
\[d(dn\circ Q,x)(X,Y) = (dn\circ Q(Y), d_{X}x) - (dn\circ Q(X), d_{Y}x)=0,\]
since $(dx,dn\circ Q)$ is symmetric. Therefore, $\eta$ is closed. Moreover, 
\[ q(X,Y) = (dn,dn\circ Q) = tr(\sigma\to \eta(X)d_{Y}\sigma) \]
is non-degenerate and $\eta \mathfrak{q}_{\infty}=0$. Hence, $f$ is an $\Omega$-surface and for some $\tau\in \Gamma (\wedge^{2}f)$, $p(t):=\exp(t\tau)\mathfrak{q}_{\infty}$ is a linear conserved quantity of the middle pencil satisfying 
\[(p(t),p(t))=(\mathfrak{q}_{\infty},\mathfrak{q}_{\infty})=0.\]

Conversely, suppose that $f$ is a special $\Omega$-surface of type 1 whose linear conserved quantity $p$ of the middle pencil satisfies $(p(t),p(t))=0$. Let $\mathfrak{q}_{\infty}:= p_{0}$. Then $\mathfrak{q}_{\infty}$ is a space form vector for a space form with vanishing sectional curvature. Furthermore, by Corollaries~\ref{cor:degreedrop} and~\ref{cor:polynormdeg}, one of the isothermic sphere congruences, without loss of generality $s^{+}$, takes values in $\langle \mathfrak{q}_{\infty}\rangle^{\perp}$. Let $\mathfrak{p}\in \langle \mathfrak{q}_{\infty}\rangle^{\perp}$ be a point sphere complex with $|\mathfrak{p}|^{2}=-1$ and let $\mathfrak{t}\in\Gamma s^{+}$ be the lift of $s^{+}$ such that $(\mathfrak{t},\mathfrak{p})=-1$. Then $\mathfrak{t}$ defines a tangent plane congruence for the space form projection $\mathfrak{f}:\Sigma\to\mathfrak{Q}^{3}$ of $f$. Now $\eta^{+}$ has the form 
\[ \eta^{+} = \mathfrak{t}\wedge d\mathfrak{t}\circ Q,\]
for some $Q\in \Gamma End(T\Sigma)$. Therefore, 
\[ q(X,Y) = tr(\sigma\mapsto \eta^{+}_{X}d_{Y}\sigma)=(d\mathfrak{t},d\mathfrak{t}\circ Q),\]
and $q$ is holomorphic with respect to the conformal structure induced by $\mathfrak{t}$. Furthermore, since $\eta^{+}$ is closed, we have that
\begin{align*}
0 = (d\eta^{+}(X,Y))\mathfrak{f} &= (d\mathfrak{t}\curlywedge d\mathfrak{t}\circ Q + \mathfrak{t}\wedge d(d\mathfrak{t}\circ Q))(X,Y)\mathfrak{f}\\
&= - (d(d\mathfrak{t}\circ Q)(X,Y),\mathfrak{f}) \mathfrak{t}\\
&=( (d\mathfrak{t}\circ Q(Y),d_{X}\mathfrak{f}) -  (d\mathfrak{t}\circ Q(X),d_{Y}\mathfrak{f}))\mathfrak{t}.
\end{align*}
Thus, $q$ commutes with the second fundamental form of $\mathfrak{f}$. Hence, $\mathfrak{f}$ projects to an $L$-isothermic surface. 

We therefore have the following theorem:

\begin{theorem}
\label{thm:liso}
Special $\Omega$-surfaces of type 1 whose linear polynomial conserved quantity $p$ satisfies $(p(t),p(t))=0$ are the $L$-isothermic surfaces of any Laguerre geometry defined by $p(0)$. 
\end{theorem}

\subsubsection{Calapso transforms}
$L$-isothermic surfaces are well known to be the deformable surfaces of Laguerre geometry~\cite{MN1996} and this gives rise to $T$-transforms for these surfaces~\cite{MN2014}. Therefore it is unsurprising that the Calapso transforms of Legendre lifts of $L$-isothermic surfaces yield $L$-isothermic surfaces.
 
Fix $t\in\mathbb{R}$ and let $f^{t}$ be a Calapso transform of $f$. Then by Proposition~\ref{prop:calpcq}, the middle pencil of $f^{t}$ admits a linear conserved quantity $p^{t}$ defined by 
\[ p^{t}(s)= T(t)p(t+s).\]
Since $T(t)$ takes values in $\textrm{O}(4,2)$ we have that 
\[ (p^{t}(s),p^{t}(s)) = (p(t+s),p(t+s)) = 0.\]
Now $T(t)p(t)$ is a constant null vector. Therefore, by premultiplying by an appropriate Lie sphere transformation, we may assume that it is $\mathfrak{q}_{\infty}$.  By applying Theorem~\ref{thm:liso} we obtain the following theorem:

\begin{theorem}
The Calapso transforms of $L$-isothermic surfaces are $L$-isothermic. 
\end{theorem}

\subsubsection{The Bianchi-Darboux transform}

Suppose that $\hat{f}$ is a Darboux transform of $f$ with parameter $m$ and suppose that $\hat{s}\le \hat{f}$ is the parallel subbundle of $d+m\eta^{mid}$. Then by Proposition~\ref{prop:darbpoly}, if $p(m) \in \Gamma \hat{s}^{\perp}$, then $\hat{p}(t):= \Gamma^{\hat{s}}_{s}(1-t/m)p(t)$ is a linear conserved quantity of $d+m\hat{\eta}^{mid}$ with $(\hat{p}(t),\hat{p}(t))=(p(t),p(t))$ and $\hat{p}(0) = p(0)$. Hence, by Theorem~\ref{thm:liso}, $\hat{f}$ projects to a $L$-isothermic surface in any space form with point sphere complex $p(0)$. It was shown (via a lengthy computation) in~\cite{P2015} that this transformation coincides with the Bianchi-Darboux transformation (see for example~\cite{E1962, MN2000}):
 
\begin{theorem}
The Bianchi-Darboux transforms of an $L$-isothermic surface constitute a 4-parameter family of Darboux transforms of its Legendre lifts.
\end{theorem}

\subsubsection{Associate surface}
We shall now recover the result of~\cite[Section 6]{RSS2009} that $L$-isothermic surfaces are the Combescure transforms of minimal surfaces.

Let $x:\Sigma\to \mathbb{R}^{3}$ be an $L$-isothermic surface. Given that 
\[ \eta^{+}  =\mathfrak{t}\wedge d\mathfrak{t}\circ Q\]
for some $Q\in\Gamma End(T\Sigma)$, we have that $(\eta^{+}\mathfrak{q}_{\infty},\mathfrak{p})=0$. By comparing with Section~\ref{subsec:asssurf}, we have that there is an associate Gauss map $\hat{x}$  of $x$ satisfying
\[ 0 = \frac{1}{\hat{\kappa}_{1}} + \frac{1}{\hat{\kappa}_{2}} = \frac{\hat{\kappa}_{1}+\hat{\kappa}_{2}}{\hat{\kappa}_{1}\hat{\kappa}_{2}}.\]
Thus, there exists a minimal surface $\hat{x}$ with the same spherical representation as $x$. In fact, we have a converse to this result: 

\begin{theorem}
Suppose that $\hat{x}:\Sigma\to\mathbb{R}^{3}$ is a minimal surface. Then any Combescure transform $x:\Sigma\to \mathbb{R}^{3}$ of $\hat{x}$ is an $L$-isothermic surface.
\end{theorem} 
\begin{proof}
Let $x$ be a Combescure transformation of $\hat{x}$, i.e., $x$ and $\hat{x}$ have the same spherical representation. Let $n$ be the common normal of these surfaces. Then the result follows by the fact that 
\[ \eta:= (n+\mathfrak{p} + (n,x)\mathfrak{q}_{\infty})\wedge (d\hat{x} + (d\hat{x},x)\mathfrak{q}_{\infty})\]
is a closed 1-form.
\end{proof}

The characterisation of $L$-isothermic surfaces as the Combescure transforms of minimal surfaces shows that the class of $L$-isothermic surfaces is preserved by Combescure transformation.

\subsection{Further work}
There is one case that we have not considered in this section - when $f$ admits a linear conserved quantity $p$ such that $(p(t),p(t))$ is a linear polynomial with vanishing constant term. It would be interesting to know if these surfaces have a classical interpretation in the Laguerre geometry defined by $p(0)$. One interesting fact about these surfaces is that if we further project into a Euclidean subgeometry of $\langle p(0)\rangle^{\perp}$ then the resulting surface is an associate surface of itself.
Furthermore, by Remark~\ref{rem:guichroot} these surfaces appear as one of the Calapso transforms of a Guichard surface.

\subsection{Complementary surfaces}
Suppose that $f$ is a special $\Omega$-surface of type 1 with linear conserved quantity $p(t)=p_{0}+tp_{1}$. 
Now the polynomial $(p(t),p(t))$ has degree less than or equal to 1 and admits non-zero roots if and only if either
\begin{itemize}
\item $(p(t),p(t))$ is linear with non-zero constant term, in which case $f$ projects to a Guichard surface in $\langle p(0)\rangle^{\perp}$, by Theorem~\ref{thm:guichard}, or 
\item $(p(t),p(t))$ is the zero polynomial, in which case $f$ projects to an $L$-isothermic surface in the Laguerre geometry defined by $p(0)$, by Theorem~\ref{thm:liso}.
\end{itemize}
Now suppose that $m$ is a root of $p(m)$ and let $\hat{f}$ be the corresponding complementary surface. Now by Theorem~\ref{prop:polymid}, $p_{1}\in\Gamma f$ and thus 
\[ f+\hat{f} = f\oplus\langle p(0)\rangle.\]

Conversely, suppose that $\hat{f}$ is a Darboux transform of $f$ with parameter $m$ such that there exists a constant vector $\mathfrak{q}\in\Gamma (f+\hat{f})$. Let $\hat{\sigma}\in\Gamma \hat{f}$ be a parallel section of $d+m\eta^{mid}$. Now 
\[ \hat{\sigma} = \lambda \, \mathfrak{q} + \sigma\]
for some non-zero smooth function $\lambda$ and $\sigma\in\Gamma f$. Thus,
\[ 0 =  (d+m\eta^{mid})\hat{\sigma} = d\lambda\, \mathfrak{q} + d\sigma + m\lambda\eta^{mid}\mathfrak{q}.\]
Since $\mathfrak{q}$ never belongs to $f$ and $\mathfrak{q}$ belongs to $f+\hat{f}$, we have that $\mathfrak{q}$ never belongs to $f^{\perp}$. Thus, 
\[d\lambda=0\quad \text{and}\quad d\sigma + m\lambda\eta^{mid}\mathfrak{q} = 0.\]
Therefore, $d+t\eta^{mid}$ admits a linear conserved quantity $p$ defined by
\[ p(t) = m\lambda \mathfrak{q} + t\sigma.\]
Furthermore, using that $\hat{\sigma}$ is lightlike, we have that 
\[ (p(t),p(t)) = m\lambda  (m\lambda  |\mathfrak{q}|^{2} + 2t(\sigma,\mathfrak{q})) = m\lambda^{2} ( m-t)|\mathfrak{q}|^{2}.\]
Therefore, $p$ admits non-zero roots and $\hat{f}$ is complementary surface of $f$ with respect to $p$.

If $(p(0),p(0))$ is non-zero then 
\[ f\cap\langle p(0)\rangle^{\perp} = f\cap\hat{f} = \hat{f}\cap\langle p(0)\rangle^{\perp}.\]
Hence, $f$ and $\hat{f}$ project to the same Guichard surface in the conformal geometry $\langle p(0)\rangle^{\perp}$. If $(p(0),p(0))=0$ then, by Corollary~\ref{cor:legperp}, $p(0)$ lies nowhere in $f$ and we must have that $p(0)\in\Gamma \hat{f}$. Thus, $\hat{f}$ is totally umbilic. 

We have thus arrived at the following theorem:

\begin{proposition}
\label{prop:darbchar}
Suppose that $\hat{f}$ is a Darboux transform of $f$. Then there exists a constant vector $\mathfrak{q}\in\Gamma(f+\hat{f})$ if and only if $f$ is a type 1 special $\Omega$-surface that admits $\hat{f}$ as a complementary surface. Furthermore, if $\mathfrak{q}$ is lightlike then $f$ projects to an $L$-isothermic surface in the Laguerre geometry defined by $\mathfrak{q}$ and $\hat{f}$ is totally umbilic. Otherwise, $f$ and $\hat{f}$ project to the same Guichard surface in the conformal geometry $\langle \mathfrak{q}\rangle^{\perp}$. 
\end{proposition}

In particular, Proposition~\ref{prop:darbchar} gives us a characterisation of $L$-isothermic surfaces in terms of their Darboux transforms:

\begin{theorem}
An $\Omega$-surface projects to an $L$-isothermic surface in some Laguerre geometry if and only if it admits a totally umbilic Darboux transform.
\end{theorem}

\section{Linear Weingarten surfaces}
Let $\mathfrak{f}:\Sigma\to\mathfrak{Q}^{3}$ be the space form projection of a Legendre map $f:\Sigma\to\mathcal{Z}$ into the (Riemannian or Lorentzian) space form $\mathfrak{Q}^{3}$ with constant sectional curvature $\kappa$. Let $\varepsilon$ be $+1$ in the case that $\mathfrak{Q}^{3}$ is a Riemannian space form and $-1$ in the case that it is Lorentzian. Recall the following definition: 
\begin{definition}
Where $\mathfrak{f}$ immerses we say that it is a linear Weingarten surface if
\begin{equation}
\label{eqn:linwein}
a K + 2b H + c = 0
\end{equation}
for some $a,b,c\in\mathbb{R}$, not all zero, where $K=\kappa_{1}\kappa_{2}$ is the extrinsic Gauss curvature of $\mathfrak{f}$ and $H =\frac{1}{2}(\kappa_{1}+\kappa_{2})$ is the mean curvature of $\mathfrak{f}$. 
\end{definition}
A special case of linear Weingarten surfaces is given by flat fronts:
\begin{definition}
A surface $\mathfrak{f}:\Sigma\to\mathfrak{Q}^{3}$ is a flat front if, where it immerses, the intrinsic Gauss curvature $K_{int}:=\varepsilon \, K + \kappa$ vanishes.
\end{definition}

In \cite{BHR2010} it was shown that flat fronts in hyperbolic space are those $\Omega$-surface whose isothermic sphere congruences each envelop a fixed sphere. In \cite{BHR2012} it was shown that linear Weingarten surfaces in space forms correspond to Lie applicable surfaces whose isothermic sphere congruences take values in certain linear sphere complexes. This theory was discretised in~\cite{BHR2014}. In this section we shall review this theory in terms of linear conserved quantities of the middle pencil.

Recall from Section~\ref{sec:symbreak} how we break symmetry from Lie sphere geometry to space form geometry. Let $\mathfrak{q},\mathfrak{p}\in\mathbb{R}^{4,2}$ be a space form vector and point sphere complex for a space form 
\[\mathfrak{Q}^{3}:= \{ y\in \mathcal{L}: (y,\mathfrak{q})=-1, (y,\mathfrak{p})=0\}.\]
Now assume that $|\mathfrak{p}|^{2}=\pm 1$. Then $\varepsilon = -|\mathfrak{p}|^{2}$. and $\kappa =-|\mathfrak{q}|^{2}$. 

Let $f:\Sigma\to Z$ be a Legendre map and assume that $f$ projects into $\mathfrak{Q}^{3}$. Let $\mathfrak{f}:\Sigma\to \mathfrak{Q}^{3}$ denote the space form projection of $f$ and let $\mathfrak{t}:\Sigma\to\mathfrak{P}^{3}$ denote its tangent plane congruence.

Similarly to~\cite{BHR2014}, we have an alternative characterisation of the linear Weingarten condition: 

\begin{proposition}
\label{prop:weincurv}
$\mathfrak{f}$ is a linear Weingarten surface satisfying~(\ref{eqn:linwein}) if and only if 
\[ [W](s_{1},s_{2}) = [W(s_{1},s_{2})]=0,\]
where $[W]\in \mathbb{P}(S^{2}\mathbb{R}^{4,2})$ is defined by
\[ W:= a\, \mathfrak{q}\odot \mathfrak{q} + 2b \, \mathfrak{q}\odot \mathfrak{p} + c\, \mathfrak{p}\odot \mathfrak{p},\]
and $s_{1},s_{2}\le f$ are the curvature spheres of $f$. 
\end{proposition}
\begin{proof}
One can easily deduce this result by using the lifts 
\[ \sigma_{1} = \mathfrak{t} + \kappa_{1}\mathfrak{f} \quad \text{and} \quad \sigma_{2} = \mathfrak{t} + \kappa_{2}\mathfrak{f}\]
of the curvature spheres. 
\end{proof}

From Proposition~\ref{prop:weincurv} one quickly deduces the observation of~\cite{BHR2012}, that if $f$ projects to a linear Weingarten surface in a space form with space form vector $\mathfrak{q}$ and point sphere complex $\mathfrak{p}$ then $f$ projects to a linear Weingarten surface in any other space form with space form vector and point sphere complex chosen from $\langle \mathfrak{q},\mathfrak{p}\rangle$. 

\subsection{Linear Weingarten surfaces in Lie geometry}
We shall now recover the results of~\cite{BHR2012} regarding the Lie applicability of umbilic-free linear Weingarten surfaces. 

\begin{proposition}
\label{prop:weinmid}
$\mathfrak{f}$ is an umbilic-free linear Weingarten surface satisfying~(\ref{eqn:linwein})
if and only if $f$ is a Lie applicable surface with middle potential
\[ \eta^{mid} = c\, \mathfrak{f}\wedge d\mathfrak{f} - b\,(\mathfrak{f}\wedge d\mathfrak{t}+ \mathfrak{t}\wedge d\mathfrak{f}) + a\, \mathfrak{t}\wedge d\mathfrak{t}\]
and quadratic differential 
\[q = -c (d\mathfrak{f},d\mathfrak{f}) + 2b(d\mathfrak{f},d\mathfrak{t})- a(d\mathfrak{t},d\mathfrak{t}).\]
Furthermore, tubular linear Weingarten surfaces give rise to $\Omega_{0}$-surfaces and non-tubular linear Weingarten surfaces give rise to $\Omega$-surfaces whose isothermic sphere congruences are real in the case that $b^{2}-ac>0$ and complex conjugate in the case that $b^{2}-ac<0$.  
\end{proposition}
\begin{proof}
Let 
\[ \eta := c\, \mathfrak{f}\wedge d\mathfrak{f} - b\,(\mathfrak{f}\wedge d\mathfrak{t}+ \mathfrak{t}\wedge d\mathfrak{f}) + a\, \mathfrak{t}\wedge d\mathfrak{t}.\]
Then
\begin{align*}
d\eta &= c\, d\mathfrak{f}\curlywedge d\mathfrak{f} - b\, (d\mathfrak{f}\curlywedge d\mathfrak{t}+ d\mathfrak{t}\curlywedge d\mathfrak{f}) + a\, d\mathfrak{t}\curlywedge d\mathfrak{t}\\
&= (a\, K +  2b\, H + c)\, d\mathfrak{f}\curlywedge d\mathfrak{f}.
\end{align*}
Thus $\eta$ is closed if and only if $\mathfrak{f}$ is a linear Weingarten surface satisfying~(\ref{eqn:linwein}). Furthermore, one can check that, modulo $\Omega^{1}(\wedge^{2}f)$, $\eta$ is equal to
\begin{equation*}
\tfrac{1}{\kappa_{1}-\kappa_{2}}((a\kappa_{1}+b) (\mathfrak{t}+\kappa_{2}\mathfrak{f})\wedge d(\mathfrak{t}+\kappa_{2}\mathfrak{f}) -  (a\kappa_{2}+b) (\mathfrak{t}+\kappa_{1}\mathfrak{f})\wedge d(\mathfrak{t}+\kappa_{1}\mathfrak{f}) ).
\end{equation*}
Since $\mathfrak{t}+\kappa_{1}\mathfrak{f}\in\Gamma s_{1}$ and $\mathfrak{t}+\kappa_{2}\mathfrak{f}\in\Gamma s_{2}$, we have that the $\Omega^{1}(S_{1}\wedge S_{2})$ part of $\eta$ lies in $\Omega^{1}(\wedge^{2}f)$. Thus $\eta$ is the middle potential $\eta^{mid}$. 

Now the quadratic differential induced by $\eta^{mid}$ is given by 
\begin{align*} 
q &= -c (d\mathfrak{f},d\mathfrak{f}) + 2b(d\mathfrak{f},d\mathfrak{t})- a(d\mathfrak{t},d\mathfrak{t})\\ 
&= (-c-2b \kappa_{1} - a\kappa_{1}^{2})(d_{1}\mathfrak{f},d_{1}\mathfrak{f}) + (-c-2b\kappa_{2} - a\kappa_{2}^{2})(d_{2}\mathfrak{f},d_{2}\mathfrak{f}),
\end{align*}
using Rodrigues' equations, $d_{i}\mathfrak{t} +\kappa_{i}d_{i}\mathfrak{f}=0$. Since $c= -b(\kappa_{1}+\kappa_{2}) - a\kappa_{1}\kappa_{2}$, we have that 
\[ q = (\kappa_{1}-\kappa_{2})(-(a\kappa_{1}+b)(d_{1}\mathfrak{f},d_{1}\mathfrak{f}) + (a\kappa_{2}+b)(d_{2}\mathfrak{f},d_{2}\mathfrak{f})).\]
Since $\mathfrak{f}$ is an umbilic-free immersion, i.e., $\kappa_{1}\neq \kappa_{2}$, $q$ is non-zero. Moreover, 
\[ -(a\kappa_{1}+b)(a\kappa_{2}+b) = - a^{2}K - 2ab H - b^{2} = - (b^{2}-ac).\]
Therefore, $q$ is degenerate if and only if $b^{2}-ac=0$ if and only if $\mathfrak{f}$ is tubular. Furthermore, if $b^{2}-ac>0$ then $q$ is indefinite and the isothermic sphere congruences of $f$ are real, whereas if $b^{2}-ac<0$ then $q$ is positive definite and the isothermic sphere congruences of $f$ are complex conjugate. 
\end{proof}

\begin{corollary}
\label{cor:weinmidpoly}
$\mathfrak{f}$ is a linear Weingarten surface satisfying~(\ref{eqn:linwein}) if and only if  
\[
p(t) := \mathfrak{p}+ t(-b\mathfrak{f}+a\mathfrak{t})\quad\text{and}\quad
q(t) := \mathfrak{q} +t(c\mathfrak{f} - b \mathfrak{t})
\] 
are conserved quantities of the middle pencil $d+t\eta^{mid}$.
\end{corollary}

Clearly, real linear combinations of polynomial conserved quantities are polynomial conserved quantities. However, the degree of the polynomials may not be preserved. For example, one can check that there exists a constant conserved quantity within the span of the conserved quantities $p$ and $q$ of  Corollary~\ref{cor:weinmidpoly} if and only if $\mathfrak{f}$ is a tubular linear Weingarten surface, i.e., $b^{2}-ac=0$. Therefore, in the non-tubular case, any linear combination of $p$ and $q$ yields a linear conserved quantity of $d+t\eta^{mid}$. In light of this we will consider 2 dimensional vector spaces of linear conserved quantities for $\Omega$-surfaces: 

\subsubsection{Non-tubular linear Weingarten surfaces}

Suppose that $f$ is an $\Omega$-surface and suppose that $P$ is a 2 dimensional vector space of linear conserved quantities of $d+t\eta^{mid}$. By $P(t)$ we shall denote the subset of $\underline{\mathbb{R}}^{4,2}$ formed by evaluating $P$ at $t$. 

\begin{lemma}
For each $t\in\mathbb{R}$, $P(t)$ is a rank 2 subbundle of $\underline{\mathbb{R}}^{4,2}$. 
\end{lemma}
\begin{proof}
Let $p,q\in P$. Then by Lemma~\ref{lem:lincon},
\[ p(t) = \exp(-t\, \sigma^{+}\odot \sigma^{-})p_{0} \quad \text{and}\quad q(t) = \exp(-t\, \sigma^{+}\odot \sigma^{-})q_{0}, \]
for some $q_{0},p_{0}\in \mathbb{R}^{4,2}$. Then $p(t)$ and $q(t)$ are linearly dependent sections of $P(t)$ for some $t\in\mathbb{R}$ if and only if $p_{0}$ and $q_{0}$ are linearly dependent if and only if $p$ and $q$ are linearly dependent. 
\end{proof}

We may equip $P$ with a pencil of metrics $\{ g_{t}\}_{t\in\mathbb{R}\cup\{\infty\}}$ defined for each $t\in\mathbb{R}$ and $\alpha,\beta\in P$ by 
\[ g_{t}(\alpha,\beta) := (\alpha(t),\beta(t)),\]
and
\[g_{\infty} := \lim_{t\to \infty} \tfrac{1}{t}g_{t}.\]
Thus, if we write $\alpha(t) = \alpha_{0} + t \alpha_{1}$ and $\beta(t)=\beta_{0}+t \beta_{1}$ then
\[ g_{\infty}(\alpha,\beta) = (\alpha_{0},\beta_{1}) + (\beta_{0},\alpha_{1}),\]
Then, for general $t\in\mathbb{R}$, we have that 
\[ g_{t} = g_{0} + t\,g_{\infty}.\]

We shall now consider the 3-dimensional vector space $S^{2}P$ formed by the abstract symmetric product on $P$. For each $t\in\mathbb{R}$ we can identify elements of $S^{2}P$ with symmetric endomorphisms on $\mathbb{R}^{4,2}$ via the map 
\[ \phi_{t} : S^{2}P\to S^{2}P(t), \quad \alpha\odot \beta \mapsto \alpha(t)\odot \beta(t).\]
Furthermore, we have an isomorphism from $S^{2}P$ to the space of symmetric tensors on $P$ with respect to $g_{\infty}$, denoted $S^{2}_{\infty}P$ defined by 
\[ \phi_{\infty}: S^{2}P\to S^{2}_{\infty}P, \quad  \alpha\odot \beta \mapsto (\alpha\odot \beta)_{\infty},\]
where for $\gamma,\delta \in P$, 
\[  (\alpha\odot \beta)_{\infty}(\gamma,\delta) := \tfrac{1}{2}(g_{\infty}(\alpha,\gamma)g_{\infty}(\beta,\delta) + g_{\infty}(\alpha,\delta)g_{\infty}(\beta,\gamma)).\]

Using Corollary~\ref{cor:weinmidpoly}, we obtain the following proposition:

\begin{proposition}
\label{prop:lwlinspace}
Suppose that $\mathfrak{f}$ is a non-tubular linear Weingarten surface satisfying~(\ref{eqn:linwein}). Then $f$ is an $\Omega$-surface whose middle pencil admits a 2-dimensional space of linear conserved quantities $P$ with $g_{0}\neq 0$ and non-degenerate $g_{\infty}$. Furthermore, the linear Weingarten condition $[W]$ is given by $[\phi_{0}\circ \phi_{\infty}^{-1}(g_{\infty})]$.
\end{proposition}
\begin{proof}
By Proposition~\ref{prop:weinmid}, $f$ is an $\Omega$-surface and by Corollary~\ref{cor:weinmidpoly}, $P:=\langle p,q\rangle$ is a 2-dimensional space of linear conserved quantities for $d+t\eta^{mid}$, where 
\[
p(t) := \mathfrak{p}+ t(-b\mathfrak{f}+a\mathfrak{t})\quad\text{and}\quad
q(t) := \mathfrak{q} +t(c\mathfrak{f} - b \mathfrak{t}).
\] 
Since $\mathfrak{p}$ is a point sphere complex, i.e., $|\mathfrak{p}|^{2}\neq 0$, we have that $g_{0}\neq 0$. We also have that 
\[ \Delta:=g_{\infty}(p,p)g_{\infty}(q,q) - g_{\infty}(q,p)^{2} = -4(b^{2}-ac).\]
Therefore $g_{\infty}$ is non-degenerate and 
\[ \phi_{\infty}^{-1}g_{\infty} = \Delta^{-1}(g_{\infty}(p,p)q\odot q - 2\,g_{\infty}(p,q) q\odot p + g_{\infty}(q,q)p\odot p).\]  
Thus, 
\begin{align*}
\phi_{0}(\phi_{\infty}^{-1}g_{\infty}) &= \Delta^{-1}(g_{\infty}(p,p)q(0)\odot q(0) - 2\,g_{\infty}(p,q) q(0)\odot p(0) + g_{\infty}(q,q)p(0)\odot p(0))\\
&= \Delta^{-1}( - 2a\, \mathfrak{q}\odot \mathfrak{q} - 4b\, \mathfrak{q}\odot\mathfrak{p} - 2c\,\mathfrak{p}\odot \mathfrak{p})\\
&= -2\Delta^{-1}( a\, \mathfrak{q}\odot \mathfrak{q} +2b\, \mathfrak{q}\odot\mathfrak{p} +c\,\mathfrak{p}\odot \mathfrak{p}).
\end{align*}
Hence, $[W] = [\phi_{0}(\phi_{\infty}^{-1}g_{\infty})]$.
\end{proof}

\begin{remark}
It follows from the proof of Proposition~\ref{prop:lwlinspace} that if $b^{2}-ac>0$ then $g_{\infty}$ is indefinite and if $b^{2}-ac<0$ then $g_{\infty}$ is definite. Then it follows by Proposition~\ref{prop:weinmid} that the isothermic sphere congruences are real when $g_{\infty}$ is indefinite and complex conjugate when $g_{\infty}$ is definite. 
\end{remark}

We now seek a converse to Proposition~\ref{prop:lwlinspace}. Firstly we have the following technical lemma that gives conditions for our $\Omega$-surface to project to a well-defined map in certain space forms, i.e., so that our point sphere map does not have points at infinity:

\begin{lemma}
\label{lem:linwspf}
Suppose that $\mathfrak{q},\mathfrak{p}\in P(0)$ are a space form vector and point sphere complex for a space form $\mathfrak{Q}^{3}$. Then $f$ defines a point sphere map $\mathfrak{f}:\Sigma\to \mathfrak{Q}^{3}$ with tangent plane congruence $\mathfrak{t}:\Sigma\to \mathfrak{P}^{3}$ if and only if $g_{\infty}$ is non-degenerate. 
\end{lemma}
\begin{proof}
Let $p,q\in P$ such that 
\[ p(t) =  \exp(-t\, \sigma^{+}\odot \sigma^{-})\mathfrak{p} \quad \text{and}\quad q(t)=\exp(-t\, \sigma^{+}\odot \sigma^{-})\mathfrak{q}.\]
Then 
\begin{align*} g_{\infty}(p,p) &= -2(\sigma^{+},\mathfrak{p})(\sigma^{-},\mathfrak{p}),\\g_{\infty}(q,q) &= -2(\sigma^{+},\mathfrak{q})(\sigma^{-},\mathfrak{q}), \quad \text{and}\\
g_{\infty}(p,q) &= -(\sigma^{+},\mathfrak{p})(\sigma^{-},\mathfrak{q}) - (\sigma^{+},\mathfrak{q})(\sigma^{-},\mathfrak{p}).
\end{align*}
One can then deduce that 
\begin{align*} 
g_{\infty}(p,p)g_{\infty}(q,q)-g_{\infty}(p,q)^{2} &= - ((\sigma^{+},\mathfrak{p})(\sigma^{-},\mathfrak{q}) - (\sigma^{+},\mathfrak{q})(\sigma^{-},\mathfrak{p}))^{2} \\
&= - ((\sigma^{+}\wedge \sigma^{-})\mathfrak{p},\mathfrak{q})^{2}.
\end{align*}
By Corollary~\ref{cor:legperp}, $f$ lies nowhere in $\langle \mathfrak{q}\rangle^{\perp}$ or $\langle \mathfrak{p}\rangle^{\perp}$. It then follows by Lemma~\ref{lem:spfwedge} that $f$ defines a point sphere map $\mathfrak{f}$ and tangent plane congruence $\mathfrak{t}$ if and only if $g_{\infty}$ is non-degenerate.
\end{proof}

We are now in a position to state the following proposition:

\begin{proposition}
\label{prop:olinspace}
Suppose that $f$ is an umbilic-free $\Omega$-surface whose middle pencil admits a 2-dimensional space of linear conserved quantities $P$, such that $g_{0}\neq 0$ and $g_{\infty}$ is non-degenerate. Then $f$ projects to a non-tubular linear Weingarten surface with 
\[ [W]= [\phi_{0}\circ\phi^{-1}_{\infty}(g_{\infty})],\]
where it immerses, in any space form determined by space form vector and point sphere complex $\mathfrak{q},\mathfrak{p}\in P(0)$. 
\end{proposition}
\begin{proof}
Since $g_{0}\neq 0$ we may choose a space form vector $\mathfrak{q}$ and point sphere complex $\mathfrak{p}$ for a space form $\mathfrak{Q}^{3}$ from $P(0)$. By Lemma~\ref{lem:linwspf}, since $g_{\infty}$ is non-degenerate, $f$ projects to a point sphere map $\mathfrak{f}:\Sigma\to \mathfrak{Q}^{3}$ with tangent plane congruence $\mathfrak{t}:\Sigma\to \mathfrak{P}^{3}$.

Now we may choose $p,q\in P$ such that $p(0)=\mathfrak{p}$ and $q(0)=\mathfrak{q}$. By Lemma~\ref{lem:lincon}, for certain Christoffel dual lifts $\sigma^{\pm}$, $(\sigma^{\pm},\mathfrak{q})$ and $(\sigma^{\pm},\mathfrak{p})$ are constant and  
\[ p(t) = \mathfrak{p} + t (\sigma^{+}\odot \sigma^{-})\mathfrak{p} \quad \text{and}\quad 
q(t) = \mathfrak{q} + t (\sigma^{+}\odot \sigma^{-})\mathfrak{q}.\]
Therefore, there exists constants (possibly complex) $\lambda^{\pm}$ and $\mu^{\pm}$ such that 
\[ \sigma^{\pm} = \lambda^{\pm}\mathfrak{f} + \mu^{\pm}\mathfrak{t}\]
and 
\begin{align*}
p(t) &= \mathfrak{p} + t(\tfrac{1}{2}(\mu^{+}\lambda^{-} +\mu^{-}\lambda^{+})\mathfrak{f} +\mu^{+}\mu^{-}\mathfrak{t})\quad \text{and}\\
q(t) &= \mathfrak{q} + t( \lambda^{+}\lambda^{-}\mathfrak{f} +\tfrac{1}{2}(\mu^{+}\lambda^{-} +\mu^{-}\lambda^{+})\mathfrak{t}).
\end{align*}
Then, by Corollary~\ref{cor:weinmidpoly}, where it immerses, $\mathfrak{f}$ is a linear Weingarten surface satisfying~(\ref{eqn:linwein}) with 
\[ a := \mu^{+}\mu^{-}, \quad b:=- \frac{1}{2}(\mu^{+}\lambda^{-} +\mu^{-}\lambda^{+}) \quad \text{and}\quad c:= \lambda^{+}\lambda^{-}.\]
On the other hand
\[ a = -\tfrac{1}{2}(p,p)_{\infty},\quad b:= \tfrac{1}{2}(p,q)_{\infty} \quad \text{and}\quad c:= -\tfrac{1}{2}(q,q)_{\infty}.\]
Thus, 
\begin{align*}
W &= a\,\mathfrak{q}\odot\mathfrak{q} + 2b\, \mathfrak{q}\odot \mathfrak{p} + c\,\mathfrak{p}\odot \mathfrak{p} \\
&= - \tfrac{1}{2} ((p,p)_{\infty}q(0)\odot q(0) - 2(p,q)_{\infty} q(0)\odot p(0) + (q,q)_{\infty}p(0)\odot p(0))\\
&= - \tfrac{\Delta}{2}(\phi_{0}(\phi_{\infty}^{-1}g_{\infty})),
\end{align*}
where $\Delta := (p,p)_{\infty}(q,q)_{\infty}- (p,q)_{\infty}^{2}$. Furthermore, $b^{2}-ac = \Delta$. Hence, $\mathfrak{f}$ is non-tubular. 
\end{proof}

If $g_{\infty}$ is non-degenerate on $P$ then $g_{\infty}$ induces two null directions on $P$. In the case that $g_{\infty}$ is indefinite these are real directions and in the case that $g_{\infty}$ is definite they are complex conjugate. Let $q^{\pm}$ be two linearly independent vectors in $P\otimes\mathbb{C}$ and define $\mathfrak{q}^{\pm} := q^{\pm}(0) \in\mathbb{R}^{4,2}\otimes \mathbb{C}$. Then 
\[ q^{\pm}(t) = \exp(-t\, \sigma^{+}\odot \sigma^{-})\mathfrak{q}^{\pm}.\]
Thus 
\[ (q^{\pm},q^{\pm})_{\infty} = -2(\sigma^{\pm},\mathfrak{q}^{\pm})(\sigma^{\mp},\mathfrak{q}^{\pm})\]
and 
\[ (q^{+},q^{-})_{\infty} = -(\sigma^{+},\mathfrak{q}^{+})(\sigma^{-},\mathfrak{q}^{-})- (\sigma^{+},\mathfrak{q}^{-})(\sigma^{-},\mathfrak{q}^{+}).\]
Therefore, $q^{\pm}$ are null with respect to $(\, ,\,)_{\infty}$ if and only if we have (after possibly switching $q^{\pm}$) that $(\sigma^{\pm},\mathfrak{q}^{\pm})=0$, i.e., the isothermic sphere congruences $s^{\pm}$ take values in $\langle \mathfrak{q}^{\pm}\rangle^{\perp}$. Now by applying Proposition~\ref{prop:lwlinspace} and Proposition~\ref{prop:olinspace} we obtain the main result of~\cite{BHR2012}:

\begin{theorem}
Non-tubular linear Weingarten surfaces in space forms are those $\Omega$-surfaces whose isothermic sphere congruences each take values in a linear sphere complex. 
\end{theorem}

Furthermore, by scaling $q^{\pm}$ appropriately we have that $g_{\infty} =(q^{+}\odot q^{-})_{\infty}$. Therefore, we have that
\[ [W] = [\mathfrak{q}^{+}\odot \mathfrak{q}^{-}],\]
which was shown in~\cite{BHR2014} for the discrete case.

\subsubsection{Tubular linear Weingarten surfaces}
\label{subsec:tublin}
In \cite{BHR2012}, the following theorem is proved:

\begin{theorem}
\label{thm:tublinwein}
Tubular linear Weingarten surfaces in space forms are those $\Omega_{0}$-surfaces whose isothermic curvature sphere congruence takes values in a linear sphere complex. 
\end{theorem}

We shall recover this result in terms of our setup. Suppose that $\mathfrak{f}$ is a tubular linear Weingarten surface satisfying~(\ref{eqn:linwein}), i.e., $b^{2}-ac=0$. Then by Proposition~\ref{prop:weinmid}, $f$ is an $\Omega_{0}$-surface and, by Corollary~\ref{cor:weinmidpoly}, the middle pencil of $f$ admits conserved quantities 
\[
p(t) := \mathfrak{p}+ t(-b\mathfrak{f}+a\mathfrak{t})\quad\text{and}\quad
q(t) := \mathfrak{q} +t(c\mathfrak{f} - b \mathfrak{t}).
\] 
Then 
\[ \mathfrak{q}_{0}:= c\, p(t) + b\, q(t) = c\, \mathfrak{p} + b\, \mathfrak{q} + t( ac-b^{2})\mathfrak{t} = c\, \mathfrak{p} + b\, \mathfrak{q}\]
is a non-zero constant conserved quantity of $d+t\eta^{mid}$. This implies that $\eta^{mid}\mathfrak{q}_{0}=0$. Without loss of generality, assume that the middle potential has the form
\[ \eta^{mid} = \sigma_{1}\wedge \star d\sigma_{1}. \]
Then 
\[ 0 = \eta^{mid}\mathfrak{q}_{0} = (\sigma_{1},\mathfrak{q}_{0}) \star d\sigma_{1} - (\star d\sigma_{1},\mathfrak{q}_{0})\sigma_{1}.\]
Since $f$ is umbilic-free we have that $d_{2}\sigma_{1}$ does not take values in $f$ and thus $(\sigma_{1},\mathfrak{q}_{0}) =0$, i.e., $s_{1}\le \langle \mathfrak{q}_{0}\rangle^{\perp}$. 

Conversely, suppose that $f$ is an umbilic-free Legendre map such that $s_{1}\le \langle \mathfrak{q}_{0}\rangle^{\perp}$. Let $\tilde{\mathfrak{q}}_{0}\in\mathbb{R}^{4,2}$ such that the plane $\langle \mathfrak{q}_{0},\tilde{\mathfrak{q}}_{0}\rangle$ is not totally degenerate. Then let $[W]\in \mathbb{P}(S^{2}\mathbb{R}^{4,2})$ be defined by 
\[ W = \mathfrak{q}_{0}\odot \mathfrak{q}_{0}.\]
Then since $s_{1}\le  \langle\mathfrak{q}_{0}\rangle^{\perp}$ we have that 
\[ [W](s_{1},s_{2})=0.\]
Hence, by Proposition~\ref{prop:weincurv}, away from points where $f\perp \langle\mathfrak{q}_{0}\rangle$, $f$ projects to a linear Weingarten surface, where it immerses, in any space form determined by space form vector and point sphere complex chosen from $\langle\mathfrak{q}_{0},\tilde{\mathfrak{q}}_{0}\rangle$. Furthermore, since the discriminant of $W$ vanishes, such linear Weingarten surfaces are tubular. 

\begin{remark}
Since we assumed that $f$ is umbilic-free, we have that $f\not\perp \langle \mathfrak{q}_{0}\rangle$ on a dense open subset of $\Sigma$, by Lemma~\ref{lem:totumb}.
\end{remark}

\begin{remark}
Notice in the converse argument to Theorem~\ref{thm:tublinwein} that we did not have to assume that $f$ was an $\Omega_{0}$-surface. We can thus deduce that if one of the curvature sphere congruences of a Legendre map takes values in a linear sphere complex then it must be isothermic.
\end{remark}

\subsection{Transformations of linear Weingarten surfaces}
Using the identification of non-tubular linear Weingarten surfaces as certain $\Omega$-surfaces, we will apply the transformations of Subsection~\ref{subsec:trafos} to obtain new linear Weingarten surfaces. 

Let $f$ be an $\Omega$-surface whose middle pencil $d+t\eta^{mid}$ admits a 2-dimensional space of linear conserved quantities $P$, such that $g_{0}\neq 0$ and $g_{\infty}$ is non-degenerate. Then, by Proposition~\ref{prop:olinspace}, $f$ projects to linear Weingarten surfaces with linear Weingarten condition
\[ [W] = [\phi_{0}\circ\phi_{\infty}^{-1}( g_{\infty})],\]
in any space form determined by space form vector and point sphere complex chosen from $P(0)$.

\subsubsection{Calapso transformations}
In~\cite{BHR2012}, the Calapso transformation for $\Omega$-surfaces was used to obtain a Lawson correspondence for linear Weingarten surfaces. This was further investigated in~\cite{BHR2014} in the discrete setting. We shall recover this analysis in terms of linear conserved quantities of the middle pencil. 

Let $t\in\mathbb{R}$ and consider the Calapso transform $f^{t}=T(t)f$ of $f$. For each $p\in P$ we have by Proposition~\ref{prop:calpcq} that $p^{t}$ defined by $p^{t}(s)= T(t)p(t+s)$ is a linear conserved quantity of the middle pencil of $f^{t}$. Therefore, the middle pencil of $f^{t}$ admits a 2-dimensional space of linear conserved quantities $P^{t}$ defined by the isomorphism 
\[ \Psi: P\to P^{t}, \quad p\mapsto p^{t}.\]
As with $P$, we may equip $P^{t}$ with a pencil of metrics $\{ g^{t}_{s}\}_{s\in\mathbb{R}\cup\{\infty\}}$. Then for each $s\in\mathbb{R}$ and $\alpha^{t},\beta^{t} \in P$, 
\[ g^{t}_{s}(\alpha^{t},\beta^{t}) = (T(t)\alpha(t+s),T(t)\beta(t+s)) = (\alpha(t+s),\beta(t+s)) = g_{t+s}(\alpha,\beta),\]
by the orthogonality of $T(t)$. Thus, $\Psi$ is an isometry from $(P,g_{t+s})$ to $(P^{t},g^{t}_{s})$. It is then clear that $\Psi$ is an isometry from $(P,g_{\infty})$ to $(P^{t},g^{t}_{\infty})$. Therefore, $g^{t}_{\infty}$ is non-degenerate, and $g^{t}_{0}\neq 0$ if and only if $g_{t}\neq 0$. 

\begin{proposition}
\label{prop:flatfront}
There exists $t\in\mathbb{R}^{\times}$ such that $g_{t}=0$ if and only if $f$ projects to a flat front in any space form determined by $P(0)$. 
\end{proposition}  
\begin{proof}
Since 
\[ g_{t} = g_{0} + tg_{\infty},\]
for each $t\in\mathbb{R}^{\times}$, we have that $g_{t}=0$ if and only if $g_{0}=-tg_{\infty}$. Now let $q,p\in P$ be an orthogonal basis with respect to $g_{\infty}$. Then $[W]$ is given by
\[ W = \phi_{0}\circ \phi^{-1}_{\infty}(g_{\infty}) = \frac{1}{g_{\infty}(q,q)}q(0)\odot q(0)  + \frac{1}{g_{\infty}(p,p)}p(0)\odot p(0).\]
Thus, if $g_{0}=-tg_{\infty}$ then $q(0)$ and $p(0)$ are orthogonal and define a space form vector and point sphere complex for a space form with sectional curvature $\kappa = - g_{0}(q,q)$ and assuming that $p$ is normalised such that $g_{0}(p,p)=\pm1$, $\varepsilon = -g_{0}(p,p)$. Furthermore, by Proposition~\ref{prop:weincurv}, $f$ projects to a surface $\mathfrak{f}$ with constant extrinsic Gauss curvature 
\[ K = - \frac{g_{\infty}(q,q)}{g_{\infty}(p,p)} = -\frac{g_{0}(q,q)}{g_{0}(p,p)}=-\frac{\kappa}{\varepsilon}\]
in this space form, i.e., $\mathfrak{f}$ is a flat front.

Conversely, suppose $f$ projects to a flat front $\mathfrak{f}$ in a space form defined by space form vector $\mathfrak{q}$ and point sphere complex $\mathfrak{p}$, i.e., $\mathfrak{f}$ satisfies 
\[ \varepsilon K + \kappa= 0.\]
Since $\kappa = -|\mathfrak{q}|^{2}$ and $\varepsilon = - |\mathfrak{p}|^{2}$, by Corollary~\ref{cor:weinmidpoly} we have that
\[ p(t) = \mathfrak{p} + t\,|\mathfrak{p}|^{2}\mathfrak{t} \quad \text{and} \quad q(t)= \mathfrak{q} + t\,|\mathfrak{q}|^{2}\mathfrak{f}\]
are linear conserved quantities of the middle pencil. Moreover,
\[ g_{\frac{1}{2}} (p,p)= g_{\frac{1}{2}}(q,p)=g_{\frac{1}{2}} (q,q)=0 .\]
Hence, $g_{\frac{1}{2}}=0$.
\end{proof} 

Now consider the maps 
\[ \phi^{t}_{s}:S^{2}P^{t} \to S^{2}P^{t}(s),\quad \alpha^{t}\odot \beta^{t} \mapsto \alpha^{t}(s)\odot \beta^{t}(s).\]
Then, by extending the action of $\Psi$ to $S^{2}P$ and $T(t)$ to $S^{2}\underline{\mathbb{R}}^{4,2}$ in the standard way, one has that
\[ \phi^{t}_{s} = T(t)\circ \phi_{t+s}\circ\Psi^{-1}.\]
Furthermore, if we define $\phi^{t}_{\infty}:S^{2}P^{t}\to S^{2}_{\infty}P^{t}$ analogously to $\phi_{\infty}$, then as $g^{t}_{\infty}$ is isometric to $g_{\infty}$ via $\Psi$, we have that
\[ (\phi^{t}_{\infty})^{-1}g^{t}_{\infty} = \Psi\circ\phi_{\infty}^{-1}g_{\infty}.\]
Applying Proposition~\ref{prop:olinspace}, we have proved the following proposition:

\begin{proposition}
Suppose that $g_{t}\neq 0$. Then $f^{t}$ projects to a linear Weingarten surface with linear Weingarten condition 
\[ [W^{t}] = [T(t)\phi_{t}(\phi_{\infty}^{-1}g_{\infty})],\]
in any space form determined by space form vector and point sphere complex chosen from $P^{t}(0) =T(t)P(t)$.
\end{proposition}

In a similar way to~\cite{BHR2010}, we have the following result regarding Calapso transforms of flat fronts:

\begin{corollary}
Suppose that $f$ projects to a flat front and let $t\in\mathbb{R}$ such that $g_{t}\neq 0$. Then $f^{t}$ projects to a flat front in any space form determined by space form vector and point sphere complex chosen from $P^{t}(0) =T(t)P(t)$.
\end{corollary}
\begin{proof}
Recall that for any $s\in\mathbb{R}$, $g^{t}_{s}$ is isometric to $g_{t+s}$ via $\Psi$. Now by Proposition~\ref{prop:flatfront}, there exists $t_{0}\in\mathbb{R}^{\times}$ such that $g_{t_{0}}=0$. Therefore, $g^{t}_{t_{0}-t}=0$ and it follows by Proposition~\ref{prop:flatfront} that $f^{t}$ projects to a flat front in $P^{t}(0)$. 
\end{proof}

To summarise this section we have the following theorem:

\begin{theorem}
Calapso transforms give rise to a Lawson correspondence for non-tubular linear Weingarten surfaces.  
\end{theorem}

\subsubsection{Darboux transformations}
Suppose that $\hat{f}$ is an umbilic-free Darboux transform of $f$ with parameter $m$. Let $\hat{s}\le \hat{f}$ be the parallel subbundle of $d+m\eta^{mid}$ and let $s\le f$ be the parallel subbundle of $d+m\hat{\eta}^{mid}$. By Proposition~\ref{prop:darbpoly}, if $p$ is a linear conserved quantity of $d+t\eta^{mid}$ and $p(m)\in\Gamma \hat{s}^{\perp}$, then $\hat{p}$ is a linear conserved quantity of $d+t\hat{\eta}^{mid}$, where 
\[ \hat{p}(t) = \Gamma^{\hat{s}}_{s}(1-t/m)p(t).\]
Furthermore, $\hat{p}(0)=p(0)$ and $(\hat{p}(t),\hat{p}(t))=(p(t),p(t))$. 
Now, if we assume that $P(m)\le \hat{s}^{\perp}$, then $\hat{P}$ is a 2-dimensional space of linear conserved quantities of the middle pencil of $\hat{f}$, where we define $\hat{P}$ via the isomorphism
\[ \Upsilon :P\to \hat{P}, \quad p\mapsto \hat{p}. \]
For each $t\in\mathbb{R}$ we shall let $\Upsilon_{t}$ denote the induced isomorphism between the subbundles $P(t)$ to $\hat{P}(t)$. Then $\hat{P}(0)=P(0)$ and $\Upsilon_{0} = id_{P(0)}$. Furthermore, if we let $\{\hat{g}_{t}\}_{t\in\mathbb{R}\cup\{\infty\}}$ denote the pencil of metrics on $\hat{P}$, then, as $(\hat{p}(t),\hat{p}(t))=(p(t),p(t))$, we have that $(P,g_{t})$ is isometric to $(\hat{P},\hat{g}_{t})$ via $\Upsilon$ for all $t\in\mathbb{R}\cup\{\infty\}$. In particular, we have that $\hat{g}_{0}\neq 0$ and $\hat{g}_{\infty}$ is non-degenerate. Therefore, by Theorem~\ref{prop:olinspace}, $\hat{f}$ projects to linear Weingarten surfaces in any space form defined by space form vector and point sphere complex chosen from $\hat{P}(0)=P(0)$.

As for $f$, define $\hat{\phi}_{t}:S^{2}\hat{P}\to S^{2}\hat{P}(t)$ and $\hat{\phi}_{\infty}:S^{2}\hat{P}\to S^{2}_{\infty}\hat{P}$ accordingly. Then for each $t\in\mathbb{R}$,
\[ \hat{\phi}_{t} = \Upsilon_{t}\circ\phi_{t}\circ \Upsilon^{-1}\]
and $\hat{\phi}_{0} = \phi_{0}$. Furthermore, 
\[ \hat{\phi}^{-1}_{\infty} (\hat{g}_{\infty}) = \phi^{-1}_{\infty}(g_{\infty}).\]
Then the linear Weingarten condition for $\hat{f}$ is given by 
\[ [\widehat{W}] = [\hat{\phi}_{0}\circ \hat{\phi}^{-1}_{\infty} (\hat{g}_{\infty})]=[\phi_{0}\circ \phi^{-1}_{\infty} (g_{\infty})].\]
Therefore, we have proved the following proposition:

\begin{proposition}
$\hat{f}$ is a linear Weingarten surface with the same linear Weingarten condition as $f$ in any space form determined by space form vector and point sphere complex chosen from $P(0)$.
\end{proposition}

By Lemma~\ref{lem:darbcons} one deduces that, for each $m\in\mathbb{R}^{\times}$, there exists a 2-parameter family of Darboux transforms with parameter $m$ such that $P(m)\le s^{\perp}$. Therefore:

\begin{theorem}
\label{thm:lindarb}
A linear Weingarten surface possesses a 3-parameter family of Darboux transforms that satisfy the same linear Weingarten condition as the initial surface. 
\end{theorem}

In~\cite[\S 273]{B1922} and~\cite[\S 398]{B1903}, Bianchi constructs a 3-parameter family of Ribaucour transformations of pseudospherical ($K=-1$) surfaces in Euclidean space into surfaces of the same kind by performing two successive complex conjugate B\"acklund transformations. This has been investigated recently in~\cite{GK2016}. This transformation preserves $I+I\! I \! I$. On the other hand Proposition~\ref{prop:weinmid} tells us that $I +I\! I \! I$ coincides with the quadratic differentials of the underlying $\Omega$-surfaces. By applying Proposition~\ref{prop:ribtodarb}, one deduces that Bianchi's family of transformations is included in the 3-parameter family of Darboux transforms detailed in Theorem~\ref{thm:lindarb}. 

A similar transformation exists for spherical ($K=1$) surfaces in Euclidean space. It was shown in~\cite{HP1997} and subsequently~\cite{IK2005} that these transformations are induced by the Darboux transformations of their parallel CMC-surfaces. Recalling from Theorem~\ref{thm:darbisosurf} that Darboux transforms of isothermic surfaces are Darboux transforms of their Legendre lifts, one deduces that these transformations are also included in the family detailed in Theorem~\ref{thm:lindarb}. 

\bibliographystyle{hplain}
\bibliography{bibliography2015}

\end{document}